\pdfoutput=1
\documentclass[11pt]{amsart}

\usepackage[margin=3.5cm]{geometry}
\usepackage{amsmath,amssymb,amsfonts,amsthm,amsfonts,mathrsfs}
\usepackage{epic}
\usepackage{amsopn,amscd}
\usepackage{color, transparent}             
\usepackage[usenames, dvipsnames]{xcolor}	
\usepackage[colorlinks, breaklinks,
	bookmarks = false,
	pdfstartview = Fit,  
	pdfview = Fit,       
	linkcolor = Blue,
	urlcolor = Green,
	citecolor = Green,
	hyperfootnotes = false
	]{hyperref}
\usepackage{multirow}
\usepackage[all,cmtip]{xy}
\usepackage{stmaryrd}

 1
 1
 1

\newcommand{\cross}{\times}

\newcommand{\bbZ}{\mathbb{Z}}
\newcommand{\bbQ}{\mathbb{Q}}
\newcommand{\bbR}{\mathbb{R}}
\newcommand{\bbC}{\mathbb{C}}

\newcommand{\M}{\mathscr{M}}
\newcommand{\sO}{\mathscr{O}}

\newcommand{\sL}{\mathscr{L}}

\newcommand{\sN}{\mathscr{N}}
\newcommand{\sA}{\mathscr{A}}
\newcommand{\sD}{\mathscr{D}}

\newcommand{\sI}{\mathscr{I}}

\newcommand{\isom}{\cong}
\newcommand{\tensor}{\otimes}

\newcommand{\fg}{\mathfrak{g}}

\newcommand{\fsl}{\mathfrak{sl}}
\newcommand{\fk}{\mathfrak{k}}

\newcommand{\oZ}{\tilde{Z}}

\DeclareMathOperator{\aut}{Aut}

\DeclareMathOperator{\End}{End}

\DeclareMathOperator{\gl}{GL}

\DeclareMathOperator{\Res}{Res}

\DeclareMathOperator{\Lie}{Lie}

\DeclareMathOperator{\Pic}{Pic}

\DeclareMathOperator{\id}{id}

\DeclareMathOperator{\p}{\partial}

\DeclareMathOperator{\cl}{Cl}
\DeclareMathOperator{\lie}{Lie}
\DeclareMathOperator{\sing}{Sing}

\theoremstyle{plain}
\newtheorem{thm}{Theorem}[section]
\newtheorem{prop}[thm]{Proposition}

\newtheorem{cor}[thm]{Corollary}

\newtheorem{lem}[thm]{Lemma}

\theoremstyle{definition} 
\newtheorem{defn}[thm]{Definition}

\theoremstyle{remark}
\newtheorem{remark}[thm]{Remark}

\newtheorem{example}[thm]{Example}

\author[T.-J.~Lee]{Tsung-Ju~Lee}
\address{T.-J.~Lee: Department of Mathematics,
National Taiwan University, Taipei 10617, Taiwan}
\email{f97221051@ntu.edu.tw}

\author[H.-W.~Lin]{Hui-Wen~Lin}
\address{H.-W.~Lin: Department of Mathematics and Taida
Institute for Mathematical Sciences (TIMS),
National Taiwan University, Taipei 10617, Taiwan}
\email{linhw@math.ntu.edu.tw}

\title{Tautological systems under the conifold transition on $G(2, 4)$}

\begin{document}
\maketitle
\begin{abstract}
Via a natural degeneration of Grassmannian manifolds $G(k,n)$ to Gorenstein toric Fano varieties $P(k,n)$ with conifold singularities, we suggest an approach to study the relation between the tautological system on $G(k,n)$ and the extended GKZ system on the small resolution $\hat{P}(k,n)$ of $P(k,n)$. We carry out the simplest case $(k,n)=(2,4)$ to ensure its validity and show that the extended GKZ system can be regarded as a tautological system on $\hat{P}(2,4)$.  

\end{abstract}

\setcounter{section}{-1}

\section{Introduction} 
\subsection{Motivation}
Mirror symmetry from physics makes numerous beautiful predictions in algebraic geometry, not only on the enumerative geometry (curve counting theory) but also some intrinsic structures of certain moduli spaces.  For a Calabi-Yau manifold $X$, the theory connects symplectic geometry ($A$ model) and complex geometry ($B$ model) of its mirror. More precisely, given a Calabi-Yau  $X$, there exists another Calabi-Yau $X'$ such that $A(X)\cong B(X')$ and $B(X)\cong A(X')$. Here, the $A$ and $B$ model theories are taken to be the genus 0 Gromov-Witten theory and the variations of Hodge structures.

Many examples are known. The first explicit mirror pairs were written down by Greene and Plesser \cite{GP1990}, the Fermat quintic. Batyrev generalized their construction to the case of Calabi-Yau hypersurfaces in toric varieties, which relies on the use of the ``reflexive polytope" \cite{B1994}. Later, Batyrev and Borisov gave a construction in the case of Calabi-Yau complete intersections in toric varieties  using ``nef-partitions" \cite{BB1994}.

M. Ried proposed a conjecture in 1987 that all connected components of moduli spaces of Calabi-Yau 3-folds can be connected via extremal transitions. An extremal transition is defined by the following diagram:
\[
\xymatrix{
& Y \ar[d]^\psi \\
& \overline{X} \ar@{~>}[r]^-i & \mathfrak{X} \supset \mathfrak{X}_t=X
}
\]
where $\psi$ is a crepant resolution and $i$ is a smoothing of singularities. Note that there is a ``topological surgery" from $Y$ to $X$ and their topological types may change. Among them the conifold transition is the most fundamental one, in which the singularities on $\overline{X}$ are only ODPs (=ordinary double points). In terms of a conifold transition $Y\mapsto X$, we got the chance to find the mirror of $X$ by means of the mirror of $Y$. 

Explicitly, Batyrev, Ciocan-Fontanine, Kim and Straten (\cite{BCKS1998, BCKS2000}) provided a way to study the mirror symmetry of Calabi-Yau complete intersection  3-folds in the Grassmannian manifolds $G(k,n)$ (or more generally, the partial flag manifolds). In their construction, $G(k,n)$ can be degenerated to a certain Gorenstein Fano toric variety $P(k,n)$ with mild singularities and then we take a specific small resolution $\hat{P}(k,n)\to P(k,n)$. Now, if we pick a Calabi-Yau complete intersection 3-fold $X$ in $G(k,n)$, then we get a complete intersection 3-fold $\overline{X}$ in $P(k,n)$ and its resolution $Y$ in $\hat{P}(k,n)$. We thus have a conifold transition
\[
\xymatrix{
& \hat{P}(k,n)\supset Y \ar[d]^\psi \\
& P(k,n)\supset \overline{X}\ar@{~>}[r]^-i & \mathfrak{X} \supset \mathfrak{X}_t=X\subset G(k,n).
}
\]
Since $\hat{P}(k,n)$ is semi Fano,  the mirror construction of Batyrev-Borisov can be applied to it. Let $Y'$ be the mirror of $Y$ and it is possible
to find an appropriate specialization
of $Y'$ to conifolds, that is, $Y'$ can be degenerated to a conifold $Y_0'$. When we take a small resolution $X'\to Y_0'$ and  $X'$ is conjecturally to be the mirror of $X$ and the corresponding diagram is the following:
\[
\xymatrix{
& X' \ar[d]\\
& Y_0' \ar@{~>}[r] & Y'.
}
\]

We are interested in the study of a mathematical proof for the mirror symmetry between $X$ and $X'$, i.e., $A(X)\cong B(X')$ and $B(X)\cong A(X')$. The Gromov-Witten theory has been well-established for smooth toric
varieties and has also been formulated for Grassmannian manifolds (\cite{B1997}). The statement $A(X)\cong B(X')$ is somehow easier and a version of it had been studied before in the literature. However the Gromov-Witten theory
on $X'$ (crepant resolution) is harder and the statement $B(X)\cong A(X')$ is completely new. Our strategy is to make a detailed study on $B(Y)$ together with the relationship between $B(X)$ and $B(Y)$. Since the toric mirror symmetry shows that $B(Y)\cong A(Y')$, we can achieve the goal by a further study on the relation between $A(Y')$ and $A(X')$. 

Recently, a closely related work has been done by Lee, Lin and Wang. 
\begin{thm}[\cite{LLW}] \label{thm1}
Let $X \nearrow Y$ be a projective conifold transition of Calabi--Yau threefolds such that $[X]$ is a nearby point of $[\bar{X}]$ in $\M_{\bar{X}}$. Then 
\begin{enumerate}
\item[(1)] $A(Y)$ can be reconstructed from a refined $A$ model of $X^{\circ} := X \setminus \bigcup_{i=1}^k S_i$ ``linked'' by the vanishing spheres $S_i$'s in $B(X)$.
\item[(2)] $B(X)$ can be reconstructed from a refined $B$ model of $Y^{\circ} := Y \setminus \bigcup_{i=1}^k C_i$ ``linked'' by the exceptional curves $C_i$'s in $A(Y)$.
\end{enumerate}
\end{thm}
The proof given in \cite{LLW} is non-constructive. Nevertheless, inspired by (2), we devote ourselves to the study of $B(Y)$ and try to restore the difference between $B(Y)$ and $B(X)$ by some contribution from the exceptional curves $C_i$'s.

\subsection{Statements of main results}
$B$ models were studied for toric varieties by using GKZ systems (\cite{GKZ}) and for Grassmannian manifolds by using Tautological systems  developed by Lian, Song and Yau (\cite{LSY2013, LY2013}). 

Gel'fand, Kapranov and Zelevinski introduced the construction of GKZ systems to govern the period integrals of toric hypersurfaces.  The idea was to use a natural basis of $H^0(X,-K_X)$, which is indexed by the integral points of $\Delta_{-K_X}$ and the relations of the integral points to write down the GKZ type binomial differential operators that can govern the period integrals. Also, the torus action on $X$ gives us additional first order operators. These operators form a holonomic system. For homogeneous spaces $X$, the standard monomial theory for representations of a reductive group gives a natural way to write down bases of cohomology of line bundles on $X$ explicitly. Hence it is natural to expect that there exists a parallel approach to construct GKZ type systems on $X$. However, unfortunately, the $D$-modules one constructs this way are almost never holonomic. There would not be enough binomial differential operators to determine the period integrals€. Thus, from our point of view, the mechanism of tautological systems can be used to resolve this problem effectively.  

More precisely, let $G$ be a fixed connected algebraic group. For every $G$-variety $X$ equipped
with a very ample $G$-equivariant line bundle $L$,
Lian, Song and Yau attached a system of differential operators defined on $H^0(X, L)^*$, depending on a given group character and showed that the system is holonomic when $X$ has finitely many $G$-orbits. Using representation theory, a generating set of the system can be written down explicitly and the constructed holonomic system governs the period integrals of Calabi-Yau hypersurfaces in a partial flag variety. (More generally, also for Calabi-Yau complete intersections in a partial flag variety.) 

We observe that the parallel construction gives rise to a kind of generalized GKZ system for Calabi-Yau hypersurfaces in a toric variety $X$ when $G$ is taken to be $\aut(X)$, and show that it coincide with the extended GKZ system introduced in \cite{HKTY1995, HLY1996}. 

\emph{From now on, the notations $X$ and $Y$ are used to denote ambient spaces, instead of hypersurfaces or complete intersections in them.}

The main purpose of this project is to determine explicitly the relation between the tautological system on $G(k,n)$ and the extended GKZ system on $\hat{P}(k,n)$. In this paper we carry out the simplest yet highly non-trivial case of $k=2$, $n=4$. 

From the extended GKZ system on $\hat{P}(2, 4)$, we can single out its subsystem which corresponds to the tautological system on $G(2,4)$.
\begin{thm}[={\bf Theorem} \ref{main1}]
The tautological system on $X:=G(2,4)$ degenerates, as a $\sD$-module, to the variant tautological (sub-)system on $Y:=\hat{P}(2,4)$.
\end{thm}
This variant tautological system is introduced in Section \ref{general TS} and is proved to be equivalent to the extended GKZ system in Thereom \ref{general on Y}.
Moreover, from the extended GKZ system on $\hat{P}(2,4)$, we can reconstruct the tautological  system on $G(2,4)$, which gives an explicit example of the statements (2) in Theorem \ref{thm1}.
\begin{thm}[={\bf Theorem} \ref{main2}]
The tautological system on $X$ is completely determined by the variant tautological system  on $Y$. Indeed, it is determined by the open part $Y^\circ:=Y-Z$, where $Z$ is the exceptional locus of the resolution $Y\to\overline{X}:=P(2,4)$.
\end{thm}

Furthermore, since the extended GKZ system can govern the period integrals only for toric hypersurfaces,  in Section \ref{TS-ci} we construct a general version of the extended GKZ system to take care of the period integrals for toric complete intersections.

As an ongoing project based on this work, we will treat the general case of arbitrary $(k,n)$ in subsequent papers.

The paper is organized as follows. In Section 1, we review some basic facts about extended GKZ systems, tautological systems and the special conifold transition considered here. In Section 2, we construct an explicit coordinate transformation from the moduli space of Calabi-Yau hypersurfaces in $G(2,4)$ to the one for $\hat{P}(2,4)$ so that we may identify the (different) group actions of $G(2,4)$ and $\hat{P}(2,4)$. In Section 3,  we give a detailed study of the possible ``variant tautological systems'' on $\hat{P}(k, n)$. In Section 4, we prove the main theorems and give remarks on certain rigidity issues to ensue the complete result for $B$ models.

\subsection{Acknowledgement}
We are grateful to Chin-Lung Wang for sharing with us his insights on this problem. H.-W. Lin is partially supported by the Ministry of Science and Technology in Taiwan and Taida Institute of Mathematical Sciences
(TIMS). 

\section{Preliminaries}

In this section, we set up the notations and recall some related well known properties.

\subsection{Extended GKZ systems}
\subsubsection{Automorphisms of Toric varieties} 
Only necessary material is recalled here, and readers are referred to \cite{C1995, C2014} for details. Let $N \cong\mathbb{Z}^n$ with the standard basis $\{e_1,\cdots,e_n\}$ and $M:=\hom_\bbZ(N,\bbZ)$. Let $\Sigma$ be a complete fan in $N_\bbR(:= N\otimes \mathbb{R})$ and $Y_\Sigma$ be the associated toric variety. We shall assume that $\Sigma$ is simplicial, i.e., $Y_\Sigma$ is $\bbQ$-factorial. Denote by $\Sigma(k)$ the collection of all $k$-dimensional cones in $\Sigma$. Put $S:=\bbC[w_\rho:\rho\in\Sigma(1)]$, which is a commutative ring graded by the (Weil) divisor class group $\cl(Y _\Sigma)$. We also use the same notation $\rho$ to denote the primitive generator of the 1-cone $\rho$. For a simplicial complete toric variety, we have the (geometric) quotient construction
$
\bbC^{\Sigma(1)}-Z(\Sigma)\slash H\isom Y_\Sigma.
$
Here $H:=\hom_\bbZ(\cl(Y_\Sigma),\bbC^\cross)$ and $Z(\Sigma)$ is the zero locus of the irrevalent ideal $I(\Sigma)$ of the fan $\Sigma$. 

To determine the automorphism group of $Y_\Sigma$, we recall that the set of the roots of $\Sigma$ is defined to be 
\[
R(\Sigma,N):=\{m\in M:\exists! \rho\in\Sigma(1)~\text{s.t.}~\langle m,\rho\rangle=-1,~\text{and}~\langle m,\rho'\rangle\geq 0~\text{for}~\rho'\neq\rho\}.
\]
It's exactly the set of the interior integral points lying on the codimension one faces of $\Delta_{-K_{Y_\Sigma}}$ (:= the polytope in $M_{\mathbb{R}}$ ($:= M\otimes \mathbb{R}$) determined by the anti-canonical divisor $-K_{Y_\Sigma}$).  Actually, each $\alpha\in R(\Sigma,N)$ gives a one-parameter subgroup of automorphisms of $Y_\Sigma$ as follows.  Let $\rho_\alpha$ be the unique 1-cone such that $\langle \alpha,\rho_\alpha\rangle=-1$ in the definition. Then, in $S\times S$,  we can associate $\alpha$ with the pair $(w_{\rho_\alpha},w^D(:=\prod_{\rho\neq \rho_\alpha} w_{\rho}^{\langle \alpha,\rho\rangle}))$. Note that $w_{\rho_\alpha}\neq w^D$ and $\deg(w_{\rho_\alpha})=\deg(w^D)$. We define, for $\lambda\in \mathbb{C}$, 
\[
y_\alpha(\lambda)(w_{\rho_\alpha})=w_{\rho_\alpha}+\lambda w^D~~{\rm and}~~y_\alpha(\lambda)(w_{\rho})=w_{\rho},~\text{for}~\rho\neq{\rho_\alpha}.
\]
This induces a graded automorphism on $S$ and thus induces an automorphism on $\bbC^{\Sigma(1)}$:
\[
y_\alpha(\lambda)_*(x_{\rho_\alpha},\mathbf{x}):=(x_{\rho_\alpha}+\lambda \mathbf{x}^D,\mathbf{x}),
\]
where $\mathbf{x}$ is a vector indexed by $\Sigma(1)-\{\rho_\alpha\}$. This automorphism preserves the Zariski open set $\bbC^{\Sigma(1)}-Z(\Sigma)$ since $\Sigma$ is assumed to be simplicial and commutes with $H$ and thus it induces an automorphism on $Y_\Sigma$. Together with the torus action, it is possible to describe the automorphism group of $Y_\Sigma$.
\begin{thm}
Let $Y_\Sigma$ be a complete simplical toric variety. Then $\aut_g(S)$, the graded automorphism of $S$, is generated by the torus $(\bbC^\cross)^{\Sigma(1)}$ and the one-parameter subgroups $y_\alpha(\lambda)$ for $\alpha\in R(\Sigma,N)$. Furthermore, we have an isomorphism
\[
\aut_g(S)\slash H\isom \aut^0(Y_\Sigma).
\]
Here $\aut^0(Y_\Sigma)$ is the connected component of the identity of the algebraic group $ \aut(Y_\Sigma)$.
\end{thm}

We remark that for the big torus $T_N\cong(\mathbb{C}^\times)^n$ in $Y$, the action (as a one-parameter subgroup with variable $\lambda$) on it is given by
\[
t_i \mapsto t_i(1+\lambda t^\alpha)^{\langle e_i,\rho_\alpha\rangle}
\]
for each root $\alpha$. Here we adapt the multi-index convention $t^\alpha:=t_1^{\alpha_1}\cdots t_n^{\alpha_n}$.

\subsubsection{Extended GKZ systems} Let $\Sigma$ be a smooth fan, that is, let $Y:=Y_\Sigma$ be nonsingular. We assume that $-K_Y$ is \textit{base point free} and fix a global holomorphic top form $\Omega$ on $Y$. The space $H^0(Y,-K_Y)^*$ parametrizes Calabi-Yau hypersurfaces in $Y$. 
Put $\Delta:=\Delta_{-K_Y}$. A general element in $H^0(Y,-K_Y)$ is of the form
\[
f=\sum_{\mu\in\Delta} b_\mu t^\mu,
\]
which defines the Calabi-Yau hypersurface $Y_f:=\{f=0\}$.
For $\alpha\in R(\Sigma,N)$ and $\lambda\in\mathbb{C}$, the automorphism $y_\alpha(\lambda)$ induces an action on $f$ and $dm:=\wedge_{i=1}^n dt_i/t_i$ explicitly by:
\begin{align}
f&\mapsto \sum_\mu b_\mu t^\mu (1+\lambda t^\alpha)^{\langle \mu,\rho_\alpha\rangle}=:f(b,\lambda),\\
\bigwedge_{i=1}^n\frac{dt_i}{t_i}&\mapsto \bigwedge_{i=1}^n \left(\frac{dt_i}{t_i}+\lambda\langle e_i,\rho_\alpha\rangle\frac{1}{1+\lambda t^\alpha}\left(\sum_{j=1}^n \alpha_j t^\alpha\cdot\frac{dt_j}{t_j}\right)\right)=:dm_\lambda.
\end{align}

We fix a smooth section $f_0$ and a cycle $\gamma\in H_{n-1}(Y_{f_0},\mathbb{Z})$. Let $\tau(\gamma)$ be a small topological $S^1$-bundle of $\gamma$ around $Y_{f_0}$. The period integral
\[
\int_\gamma \Res\frac{\Omega}{f}=\int_{\tau(\gamma)} \frac{\Omega}{f}.
\]
should be invariant under $\aut^0(Y)$-action and the scaling action.
Denote by 
\[
\Pi_\gamma(b,\lambda):=\int_{\tau(\gamma)} \frac{1}{f(b,\lambda)} dm_\lambda.
\]
By the invariant property, we have 
\begin{eqnarray}\label{eq}
\left.\frac{d}{d\lambda}\right|_{\lambda=0} \Pi_\gamma(b,\lambda)=0.
\end{eqnarray}
The differentiation can be taken into the integrand since for a \textit{cycle} $\gamma$, $\tau(\gamma)$ is homologous to $g_*\tau(\gamma)$ for $g\in\aut^0(Y)$ close to identity. 

From the equation (\ref{eq}),
\begin{align}\label{pf1}
\left.\int_{\tau(\gamma)} \frac{-f'(b,\lambda)}{(f(b,\lambda))^2}dm\right|_{\lambda=0}+\left.\int_{\tau(\gamma)} \frac{1}{f(b,\lambda)}\frac{\p}{\p\lambda}dm_\lambda\right|_{\lambda=0}=0
\end{align}
The first term is simplifies to 
\[
\int_{\tau(\gamma)} \frac{\sum_\mu -b_\mu\langle\mu,\rho_\alpha\rangle t^{\mu+\alpha}}{(f(b,0))^2}dm.
\]
For the second term, substituting 
\[
\left.\frac{\p}{\p\lambda}dm_\lambda\right|_{\lambda=0}=\sum_{i=1}^n \alpha_i\langle e_i,\rho_\alpha\rangle t^\alpha dm=\langle \alpha,\rho_\alpha\rangle t^\alpha dm,
\]
we obtain, since $\gamma$ is a \textit{cycle},
\[
(\ref{pf1})=\int_{\tau(\gamma)}\frac{\sum_\mu -b_\mu\langle \mu-\alpha,\rho_\alpha\rangle t^{\alpha+\mu}}{f(b,0)^2}=Z_\alpha \Pi_\gamma(b)
\]
where the operator $Z_\alpha$ is defined to be
\[
Z_\alpha:=\sum_{\mu\in\Delta} \langle \rho_\alpha,\mu-\alpha\rangle b_\mu\frac{\p}{\p b_{\mu+\alpha}},
\]
and thus it annihilates the period integral.
Note that the expression on the right hand side is well-defined. Indeed, if $\mu+\alpha\notin \Delta$, then there exists a $\rho\in\Sigma(1)$ such that $\langle \mu+\alpha,\rho\rangle<-1$, so $\langle \mu,\rho\rangle=\langle \alpha,\rho\rangle=-1$. We must have $\rho=\rho_\alpha$ and $\langle \mu-\alpha,\rho\rangle=0$.

\begin{defn}[cf. \cite{HLY1996}]
For a smooth projective toric variety $Y$, the extended GKZ system on $Y$ consists of the original GKZ system and the $Z_\alpha$ operators, $\alpha\in R(\Sigma,N)$. 
\end{defn}

\begin{remark}
 It has been proved that the solution space of the original GKZ system is finite-dimensional with an explicit dimension formula \cite{A1994}. In fact, this number is usually much larger than the expected dimension. It has been conjectured that the solution space to the extended GKZ system has the expected dimension, i.e., the solutions are exactly the period integrals. 
Recently, Huang, Lian, Yau and Zhu construct the full set of solutions to the GKZ system \cite{HLYZ2015}. Also the main result of this paper together with Remark \ref{SofTS} can imply that the conjecture should be modified by some contribution from the primitive cohomology.
\end{remark}

\subsection{Tautological systems}

Let $G$ be a complex connected Lie group and $X$ be a smooth projective $G$-variety. A very ample $G$-equivariant line bundle $L\to X$ gives a natural $G$-representation $V:=H^0(X,L)^*$. It gives rise to a $G$-equivariant embedding $X\hookrightarrow \mathbb{P}V$. Let $\bbC^\cross$ act on $V$ by scaling and then $V$ becomes a $\hat{G}(:=G\times\bbC^\cross)$-module. Denote by $Z:\hat{\mathfrak{g}}\to\End(V)$ the corresponding Lie algebra homomorphism. Let $\hat{X}$ be the cone over the image of $X$ under the embedding and $I:=I(\hat{X},V)\subset\bbC[V]$ be its homogeneous ideal. For each $\zeta\in V^*$, we define the derivation by $\p_\zeta(a):=\langle a,\zeta\rangle$, $a\in V$. Here $\langle-,-\rangle$ is the canonical dual pairing. Finally, let $\beta:\hat{\mathfrak{g}}\to\bbC$ be a Lie algebra homomorphism.

The tautological system is defined to be the quotient $D$-module
\[
\tau_X(G,L,\beta):=D_{V^*}/D_{V^*}\langle p(\p_\zeta):p(\zeta)\in I\rangle+D_{V^*}\langle Z(x)+\beta(x):x\in \hat{\fg}\rangle.
\]
When $L=-K_X$, $\beta=(0;\id)$ and $G\subset \aut(X)$, $\tau_X(G,L,\beta)$ annihilates the period integrals. One may use a further embedding to simplify the polynomial operators $p(\p_\zeta)$. The pay is that we have to introduce some new degree one operators. The details can be found in \cite{LSY2013, LY2013}. Here is the simplest example.
\begin{example}\label{TS-g24}
 $X=G(2,4)$, $G=SL_4$. Let $L=-K_X$ and $\beta=(0;\id)$. The tautological system $\tau_X$ is generated by the following operators.
\begin{itemize}
\item[1.] $Z(x)+\beta(x)$, $x\in \fsl_4\times\mathbb{C}$.
\item[2.] $\p_\zeta$, where $\zeta\in (H^0(X, -K_X)^*)^\perp\subset H^0(\mathbb{P}^5,\sO(4))$.
\item[3.] $\p_{\zeta_u}\p_{\zeta_v}-\p_{\zeta_p}\p_{\zeta_q}$, where $u,v,p,q\in\mathcal{E}$ and $u+v=p+q$.
\end{itemize}
Here $\mathcal{E}=\{(i_0,i_1,i_2,i_3,i_4,i_5)\in\mathbb{Z}^6_{\geq 0}:\sum_{k=0}^5 i_k=4\}$. In fact, $-K_X=\left.\sO(4)\right|_X$ under the Pl\"{u}cker embedding. We can determine the degree one operator explicitly. Let $Q_4$ be the degree 4 part in the ideal generated by the Pl\"{u}cker relations. Then $H^0(X,-K_X)=H^0(\mathbb{P}^5,\sO(4))/Q_4$ and $\zeta\in (H^0(-K_X)^*)^\perp$ if and only if $\zeta\in Q_4$. In this case, $G$ is semi-simple and $H^4(X,\bbC)_{prim}\cong\bbC$. 
\end{example}

\begin{remark}\label{SofTS}
In case $G$ is semi-simple and $X$ is a homogeneous $G$-variety, it's shown that the tautological systems govern exactly the period integrals modulo some solutions coming from the primitive cohomology of the ambient space \cite{BHLSY2014, HLZ2015}.
\end{remark}

\subsection{Degeneration to toric varieties}

In \cite{BCKS2000}, Batyrev, Ciocan-Fontanine, Kim and Straten gave an explicit construction for the degeneration of the partial flag manifolds to toric varieties. We restrict ourselves to the case  $G(2,4)$ and give a quick review of this construction.  First, consider the following ``ladder diagram" $\Lambda$. 
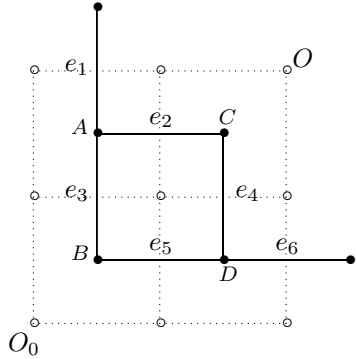
\begin{center}
\small
\unitlength=0.28mm
\begin{picture}(200,150)(-10,10)


\qbezier[40](50,0)(50,60)(50,120)
\put(48,-2){\scriptsize $\circ$}
\put(48,58){\scriptsize $\circ$}
\put(48,118){\scriptsize $\circ$}

\qbezier[40](-10,0)(-10,60)(-10,120)
\put(-12,-2){\scriptsize $\circ$}
\put(-12,58){\scriptsize $\circ$}
\put(-12,118){\scriptsize $\circ$}

\qbezier[40](110,0)(110,60)(110,120)
\put(108,-2){\scriptsize $\circ$}
\put(108,58){\scriptsize $\circ$}
\put(108,118){\scriptsize $\circ$}

\qbezier[40](-10,0)(50,0)(110,0)
\qbezier[40](-10,60)(50,60)(110,60)
\qbezier[40](-10,120)(50,120)(110,120)


\qbezier[4000](20,30)(20,90)(20,150)
\qbezier[4000](20,30)(80,30)(140,30)
\qbezier[4000](80,30)(80,60)(80,90)
\qbezier[4000](20,90)(50,90)(80,90)
\put(18,28){\scriptsize $\bullet$}
\put(18,88){\scriptsize $\bullet$}
\put(18,148){\scriptsize $\bullet$}
\put(78,28){\scriptsize $\bullet$}
\put(138,28){\scriptsize $\bullet$}
\put(78,88){\scriptsize $\bullet$}

\put(5,60){$e_3$}
\put(5,120){$e_1$}
\put(45,95){$e_2$}
\put(45,35){$e_5$}
\put(86,60){$e_4$}
\put(105,35){$e_6$}
\put(8,90){\scriptsize $A$}
\put(8,30){\scriptsize $B$}
\put(78,20){\scriptsize $D$}
\put(78,95){\scriptsize $C$}
\put(-22,-13){$O_0$}
\put(113,122){$O$}

\end{picture}
\end{center}
\quad \\

All edges $e_i$'s are oriented downward or to the right. Let $E$ be the set of all edges. We consider the free abelian group generated by the edges of this diagram, $L(E)$, and the free abelian group generated by the black dots, $L(D)$. 

The boundary map $\delta: L(E)\tensor \bbR\cong {\bbR}^6\to L(D)\tensor\bbR\cong{\bbR}^4$ is the usual boundary map in homology theories. We construct a complete fan $\Sigma$ in ${\bbR}^4$ as follows. Let $\Sigma(1)=\{\delta(e_i): i=1,\ldots,6\}$, explicitly,
\begin{align*}
&\delta (e_1)=(1,0,0,0),~~\delta (e_2)=(-1,0,1,0), ~~\delta(e_3)=(-1,1,0,0),\\
&\delta (e_4)=(0,0,-1,1),~~\delta (e_5)=(0,-1,0,1),~~ \delta(e_6)=(0,0,0,-1).
\end{align*}
The fan structures are given by the positive paths from $O$ to $O_0$. A positive path is to start from $O$ and go either downward or go to the left at each corner until reaching the vertex $O_0$. There are $6$ positive paths. Let's denote by $\pi_{12}$ the positive path crossing only $e_1$, $\pi_{13}$ the one passing through $e_2$ and $e_3$, $\pi_{14}$ the one passing through $e_2$ and $e_5$, $\pi_{23}$ the one passing through $e_3$ and $e_4$, $\pi_{24}$ the one passing through $e_4$ and $e_5$, and finally $\pi_{34}$ the one passing through $e_6$ only. 
The correspondence between positive paths and cones in $\Sigma$ is given as follows, where $w_i$ stands for the variable corresponding to the 1-cone $\delta(e_i)$ for $i=1,\ldots,6$.

\quad \\
\begin{center}
\begin{tabular}{|c|l|l|}
\hline 
  \multirow{2}{*}{Positive paths} & Maximal cones in $\Sigma$ and the &  \multirow{2}{*}{$w^{\hat{\sigma}}:=\prod_{\rho\notin \sigma(1)}w_\rho$}  \\ 
  & corresponding 1-cone generators& \\\hline 
  
  $\pi_{12}$ & $\delta(e_2),~\delta(e_3),~\delta(e_4),~\delta(e_5),~\delta(e_6)$ & $w_1$ \\\hline
  $\pi_{13}$ & $\delta(e_1),~\delta(e_4),~\delta(e_5),~\delta(e_6)$ & $w_2w_3$ \\\hline
  $\pi_{14}$ & $\delta(e_1),~\delta(e_3),~\delta(e_4),~\delta(e_6)$ & $w_2w_5$ \\\hline
  $\pi_{23}$ & $\delta(e_1),~\delta(e_2),~\delta(e_5),~\delta(e_6)$ & $w_3w_4$ \\\hline
  $\pi_{24}$ & $\delta(e_1),~\delta(e_2),~\delta(e_3),~\delta(e_6)$ & $w_4w_5$ \\\hline
  $\pi_{34}$ & $\delta(e_1),~\delta(e_2),~\delta(e_3),~\delta(e_4),~\delta(e_5)$ & $w_6$ \\\hline
\end{tabular}\\
\end{center}
\quad\\

Let $\overline{X}:=P(2,4)$ be the associated toric variety with $\Sigma$, which is a singular Gorenstein toric Fano variety. In general, the toric variety $P(k,n)$ is singular with the singular locus consisting of codimension 3 conifold strata. More explicitly, in this case, there are two singular maximal cones, namely $\{e_1,e_2,e_3,e_4,e_5\}$ and $\{e_2,e_3,e_4,e_5,e_6\}$.  The refinements 
\begin{align*}
&\{e_1,e_2,e_3,e_4,e_5\}=\{e_1,e_2,e_3,e_4\}\cup \{e_1,e_2,e_4,e_5\},\\
&\{e_2,e_3,e_4,e_5,e_6\}=\{e_2,e_4,e_5,e_6\}\cup\{e_2,e_3,e_4,e_6\}
\end{align*}
give the canonical small resolution. Denote by $Y:=\hat{P}(2,4)$ the resolution and by $\Sigma'$ the refined fan.

We shall denote the (Weil) divisor associated with $f\in E$ by $H_f$.  
The roof $\mathcal{R}$ of a ladder diagram is defined to be the set of all edges on the upper right buondary, e.g., the roof of $\Lambda$ is $\{e_1,e_2,e_4,e_6\}$. For each $e\in\mathcal{R}$, denoted by $U(e)$ the subset of $E$ consisting of the edge $e$, together with all edges $f\in E$ which are either directly below $e$ if $e$ is horizontal, or directly to the left, if $e$ is vertical.

The divisor $\sL:=\sum_{f\in U(e)} H_f$ is Cartier and independent of the choice of $e$. 
From the construction, the line bundle $\sL$ is very ample. Also, we have $-K_{P(2,4)}=4\sL$ in $\cl(P(2,4))$. The global sections of $\sL$, which are the integral points in the polytope $\Delta_{\sL}$, can be parameterized by the positive paths \cite{BCKS2000}.

\quad \\
\begin{tabular}{|c|l|l|}
\hline  
  Positive paths & Integral points in $\Delta_{\sL}$  & Vertices in $\Delta:=\Delta_{-K_Y}\subset M$ \\\hline  
  $\pi_{12}$   & $(0,0,0,0)=:v_0$    & $(3,2,2,1)=4\cdot v_0+(3,2,2,1)$ \\\hline
  $\pi_{13}$ & $(-1,0,0,0)=:v_1$   & $(-1,2,2,1)=4\cdot v_1+(3,2,2,1)$ \\\hline
  $\pi_{14}$ & $(-1,-1,0,0)=:v_2$    & $(-1,-2,2,1)=4\cdot v_2+(3,2,2,1)$ \\\hline
  $\pi_{23}$ & $(-1,0,-1,0)=:v_3$    & $(-1,2,-2,1)=4\cdot v_3+(3,2,2,1)$ \\\hline
  $\pi_{24}$ & $(-1,-1,-1,0)=:v_4$    & $(-1,-2,-2,1)=4\cdot v_4+(3,2,2,1)$ \\\hline
  $\pi_{34}$ & $(-1,-1,-1,-1)=:v_5$    & $(-1,-2,-2,-3)=4\cdot v_5+(3,2,2,1)$ \\\hline
\end{tabular}\\\\
\quad\\

There are two coordinate systems for the toric variety $\overline{X}$. One comes from the projective embedding
$\overline{X}\hookrightarrow \mathbb{P}^5$ induced by $\sL$, which corresponds to the epimorphism:
\[
\mathbb{C}[z_{12},z_{13},z_{14},z_{23},z_{24},z_{34}] \twoheadrightarrow \mathbb{C}[z_{12},z_{13},z_{14},z_{23},z_{24},z_{34}]/(z_{14}z_{23}-z_{13}z_{24}).
\]
The other one comes from the (almost geometric) quotient presentation:
\[
\overline{X}=\mathbb{C}^6-Z(\Sigma)\sslash H,
\]
where 
$
H=\{(\lambda\mu,\lambda,\mu,\lambda,\mu,\lambda\mu):\lambda,\mu\in\bbC^\cross\}\subset(\bbC^\cross)^6
$
and the irrelevant ideal is given by
\[
I(\Sigma)=\langle w_1,w_2w_3, w_2w_5, w_3w_4, w_4w_5, w_6\rangle\subset \mathbb{C}[w_1,w_2,w_3,w_4,w_5,w_6].
\] 
 From the map $\mathbb{C}^{\Sigma(1)}-Z(\Sigma) \to \overline{X}\hookrightarrow \mathbb{P}^5$, we have the natural transformation between
these two coordinates: 
\[
z_{12}=w_1, ~z_{13}=w_2w_3,~ z_{14}=w_2w_5,~ z_{23}=w_4w_3, ~z_{24}=w_4w_5,~\text{and} ~z_{34}=w_6.
\]

Recall that the Pl\"ucker embedding of $G(2,4)$ into $\mathbb{P}^5$ is given by 
\[
\mathbb{C}[z_{12},z_{13},z_{14},z_{23},z_{24},z_{34}] \twoheadrightarrow \mathbb{C}[z_{12},z_{13},z_{14},z_{23},z_{24},z_{34}]/(z_{14}z_{23}-z_{13}z_{24}+z_{12}z_{34}).
\]
Hence the degeneration of $X:=G(2,4)$ to $\overline{X}$ can be constructed through
\begin{eqnarray}\label{degcoor}
\bbC[z_{12},z_{13},z_{14},z_{23},z_{24},z_{34},t]/(z_{14}z_{23}-z_{13}z_{24}+tz_{12}z_{34}).
\end{eqnarray}
Now we get the diagram
\[
\xymatrix{
& \hat{P}(2,4) \ar[d]^\psi &\\
& P(2,4) \ar@{~>}[r]^-i & G(2,4).
}
\]

\section{Symmery operators between $G(2,4)$ and $\hat{P}(2,4)$}
To compare the tautological system on $X$ with the extended GKZ system on $Y$, the first job we have to do is to find a good transformation between  their parameter spaces, namely, $H^0(\mathbb{P}^5,\sO(4))^*$ and $H^0(Y,-K_Y)^*$. And then under such transformation, we are able to connect the symmetry operators in the tautological system on $X$ with the ones in the extended GKZ system on $Y$.
\subsection{The moduli coordinate transformation}\label{coordinates}
Let $\sL=\sum_{f\in U(e)} H_f$ which is a very ample divisor on $\overline{X}$. 
To look for a natural transformation from $H^0(Y,-K_Y)^*$ to  $H^0(\mathbb{P}^5,\sO(4))^*$, we start with the $\sL$-embedding $\psi_{\sL}: \overline{X}\hookrightarrow\mathbb{P}^5$ , $-K_{\overline{X}}=4\sL$ and the natural map
\begin{align}\label{autlin}
\Psi : H^0(\mathbb{P}^5, \sO(4))=H^0(\mathbb{P}^5,\sO(1))^{\otimes4}\to H^0(\overline{X}, -K_{\overline{X}})=H^0(Y,-K_Y).
\end{align}
The last equality holds because $\psi:Y=\hat{P}(2,4)\to \overline{X}=P(2,4)$ is a small resolution and $\overline{X}$  has rational singularities.

 Observe that $\Sigma'(1)=\Sigma(1)$ and thus we have the (Weil) divisor class group isomorphism $\cl(Y)\cong \cl(\overline{X})$. Hence, $S:=\bbC[w_\rho:\rho\in\Sigma(1)]$ can be regarded as the homogeneous coordinate ring for both $\overline{X}$ and $Y$. More explicitly,
the resolution $Y$ is a geometric quotient of $\bbC^6-Z(\Sigma')$ by \textit{the same} group $H$ with the irrelevant ideal 
\[
I(\Sigma')=\langle w_1w_3, w_1w_5, w_2w_3, w_2w_5, w_4w_3, w_4w_5, w_6w_3, w_6w_5\rangle.
\]
In fact, they all fit into the diagram
\[
\xymatrix{
& \bbC^6-Z(\Sigma') \ar[d] \ar[r]^\subset & \bbC^6-Z(\Sigma)\ar[d]\ar[rd]\\
& \bbC^6-Z(\Sigma')/H=Y \ar[r] & \bbC^6-Z(\Sigma)\sslash H=\overline{X} \ar[r]_-{\psi_{\sL}} & \mathbb{P}^5.
}
\]
and the base-point free line bundle $\psi^*\sL$ on $Y$ determines the morphism
$$
\psi_{\psi^*\sL} = \psi_{\sL}\circ\psi: \bbC^6-Z(\Sigma')/H=Y \rightarrow\mathbb{P}H^0(Y,\sL)^*=\mathbb{P}^5,
$$
which is $\aut_g(S)$-equivariant. Note that we also denote the pull-back $\psi^*\sL$ by $\sL$ when no confusion is caused. If $d=\sL\in \cl(Y)$, then $H^0(Y, \sL)\isom S_d$, the homogeneous part of $d$-grading in $S$. Thus it is fixed by the graded automorphism group $\aut_g(S)$  of $S$ and we get the following proposition.

\begin{prop}\label{linearization}
The line bundle $\mathscr{L}$ is $\aut_g(S)$-linearizable.
\end{prop}
In fact, it follows from the following lemma with $r=4$ that $H^0(Y,\sL)^{\otimes 4} \to H^0(Y,-K_Y)$ is surjective, that is, $\Psi$ is surjective.

\begin{lem}\label{lem}
The polytope $\Delta_{\sL}$ associated with $\sL$ in $M$ is normal, i.e., $r\Delta_{\sL}\cap M=r(\Delta_{\sL}\cap M)$ for all $r\in\mathbb{Z}^{>0}$.
\end{lem}
\begin{proof}
We have the well-known exact sequence
\[
0\to M\to \bbZ^6\to \cl(\overline{X})\to 0.
\]
The first map is given by $m\mapsto (\langle m,\rho_i\rangle)_{i=1}^6$, where $\rho_i$'s are the corresponding 1-cones in $\Sigma$ or $\Sigma'$. Let $D_i:=D_{\rho_i}$. It's not hard to see that the image of $D_2$, $D_3$ generate the group $\cl(\overline{X})$. Since the smooth  toric variety $Y$ comes from a fan containing a cone of maximal dimension, $\Pic(Y)\isom \cl(Y)\isom \cl(\overline{X})$ is torsion-free. Using the relation $D_1=D_2+D_3$, $D_3=D_5$, $D_2=D_4$ and $D_6=D_4+D_5$, we see that the homogeneous coordinate ring $S$ is graded by $\bbZ_{\geq 0}^2$, with 
\begin{align*}
&\deg(w_2)=\deg(w_4)=(1,0),\\
&\deg(w_3)=\deg(w_5)=(0,1),\\
&\deg(w_1)=\deg(w_6)=(1,1).
\end{align*}
By the coordinate transformation $z_{12} =w_1$, $\sL\isom \sO_{\overline{X}}(D_1)$. Set $d=[D_1]$. We claim that the ring $\bigoplus_{k=0}^\infty S_{kd}$ is a $\bbC$-algebra generated by $S_d$. Indeed, let $x\in S_{kd}$ and then $\deg(x)=(k,k)$. Suppose $x$ is a monomial. After removing $w_1$ and $w_6$, we may assume $x$ contains no variable $w_1$ and $w_6$, i.e., $x=w_2^{a_1} w_3^{a_2} w_4^{a_3} w_5^{a_4}$. By degree reason we have $a_1+a_3=a_2+a_4$. Since $S_d$ contains all the variables of degree $d=(1,1)$ in $S$, $x$ is a multiple of elements in $S_d$. This implies that $\Delta_{\sL}$ is normal.
\end{proof}

Moreover, we have to write down $\Psi$ explicitly.
According to the relation (mentioned in section 1.3) between 
the integral points in $\Delta_{\sL}$  and the vertices in $\Delta:=\Delta_{-K_Y}\subset M$, we define a map:
\[
\Phi: \mathcal{E}\to \Delta\cap M,~~I\mapsto \sum_{k=0}^5 i_k v_k+(3,2,2,1),
\]
where $\mathcal{E}=\{I=(i_0,i_1,i_2,i_3,i_4,i_5)\in\mathbb{Z}^6_{\geq 0}:\sum_{k=0}^5 i_k=4\}$. 
Then we obtain a (surjective) $\bbC$-linear map $\Phi:H^0(\mathbb{P}^5,\sO(4))\to H^0(Y,-K_Y)$ by $z^I\mapsto t^{\Phi(I)}$. The dual  (injective) map $H^0(Y, -K_Y)^* \rightarrow H^0(\mathbb{P}^4, \sO(4))^*$  gives the transformation $\Phi^t$ from the ``moduli" of Calabi-Yau hypersurfaces in $Y$ to the parameter space for tautological system on $X$ by
\begin{eqnarray}\label{coor}
\Phi^t(b_J)=\sum_{I\in\mathcal{E}, \Phi(I)=J} a_I,\quad J\mbox{ is an integral point in} \ \Delta. 
\end{eqnarray}
Here $\sum_{J\in \Delta\cap M}b_Jt^J$ stands for a general section in $H^0(Y, -K_Y)$ and $\sum_{I\in \mathcal{E}}a_Iz^I$ stands for a general section in $H^0(\mathbb{P}^5,\sO(4))$.

\subsection{Equivalence of symmetry operators}\label{symmetry}
Recall that for the tautological system on $X$, the differential operators in the first type are given by $$Z(x)+\beta(x), \quad{\rm for}\  x\in \fsl_4\times\mathbb{C}.$$
We consider the operators generated by $\fsl_4$ and note that $\beta=(0;\id)$. Since these operators are generated by group actions, it suffices to understand how the group action degenerates. We divide them into several cases.

\subsubsection{Diagonal actions}
Suppose that
\[
x=\begin{pmatrix}1 & 0 & 0 & 0\\
0 & -1 & 0 & 0 \\
0 & 0 & 0 & 0\\
0 & 0 & 0 & 0
\end{pmatrix}\in \fsl_4.
\]
It's easy to get that for $\mathbb{P}^5=\mathbb{P}(V:=\wedge^2\mathbb{C}^4)$,
\[
Z(x)=\begin{pmatrix}0 & 0 & 0 & 0 & 0 & 0\\
0 & 1 & 0 & 0 & 0 & 0\\
0 & 0 & 1 & 0 & 0 & 0\\
0 & 0 & 0 & -1 & 0 & 0\\
0 & 0 & 0 & 0 & -1 & 0\\
0 & 0 & 0 & 0 & 0 & 0\\
\end{pmatrix}^t
=\begin{pmatrix}0 & 0 & 0 & 0 & 0 & 0\\
0 & 1 & 0 & 0 & 0 & 0\\
0 & 0 & 1 & 0 & 0 & 0\\
0 & 0 & 0 & -1 & 0 & 0\\
0 & 0 & 0 & 0 & -1 & 0\\
0 & 0 & 0 & 0 & 0 & 0\\
\end{pmatrix}\in {\rm End}(V^*),
\]
so the corresponding operator $Z(x)+\beta(x)$ in the tautological system is given by
\[
\sum_{I\in\mathcal{E}} (i_1+i_2-i_3-i_4)a_I\frac{\p}{\p a_I}.
\]
Put $\lambda = e^{s}$ and then the corresponding actions on $z_{ij}$'s are given by
\[
z_{12}\mapsto z_{12},~z_{13}\mapsto \lambda z_{13},~z_{14}\mapsto \lambda z_{14}, z_{23}\mapsto \lambda^{-1} z_{23},~z_{24}\mapsto \lambda^{-1} z_{24},~z_{34}\mapsto z_{34},
\]
so it corresponds to the action on $w_i$'s via 
\[
 (w_1,w_2,w_3,w_4,w_5,w_6)\mapsto (w_1,\lambda w_2,w_3,\lambda^{-1} w_4,w_5,w_6).
\]
Let $\lambda_2={\lambda_4}^{-1}=\lambda$, $\lambda_j=1$ for other $j$.
In terms of the coordinates on $T_N$, this action corresponds to 
\[
t_i\mapsto \prod_{k=1}^6 {\lambda_k}^{\langle f_i,\delta(e_k)\rangle} t_i
\]
where $f_i$'s stand for the standard basis of $M$, that is,
\[
t_1\mapsto \lambda^{-1} t_1,~ t_2\mapsto t_2,~t_3\mapsto \lambda^{2} t_3,~ t_4\mapsto \lambda^{-1} t_4. 
\]
The corresponding GKZ-system is 
\[
\sum_{J\in\Delta\cap M} (-j_1+2j_3-j_4)b_J\frac{\p}{\p b_J}.
\]
Recall our identification between $I$ and $J$:
\[
\begin{pmatrix}
j_1\\
j_2\\
j_3\\
j_4
\end{pmatrix}
=\sum_{k=0}^5 i_kv_k+
\begin{pmatrix}
3\\
2\\
2\\
1
\end{pmatrix}=
\begin{pmatrix}
-i_1-i_2-i_3-i_4-i_5\\
-i_2-i_4-i_5\\
-i_3-i_4-i_5\\
-i_5
\end{pmatrix}
+
\begin{pmatrix}
3\\
2\\
2\\
1
\end{pmatrix}
=
\begin{pmatrix}
i_0-1\\
2-i_2-i_4-i_5\\
2-i_3-i_4-i_5\\
1-i_5
\end{pmatrix}.
\]
Then 
\begin{align*}
-j_1+2j_3-j_4&=-i_0+1+4-2i_3-2i_4-2i_5+i_5-1\\
&=4-i_0-2i_3-2i_4-i_5=i_1+i_2-i_3-i_4.
\end{align*}
Under $\Phi^t(b_J)=\sum_{I\in\mathcal{E}, \Phi(I)=J} a_I$,
\[
\frac{\p}{\p a_I}=\frac{\p}{\p b_{\Phi(I)}}.
\]
and the operator 
\[
\sum_{I\in\mathcal{E}} (i_1+i_2-i_3-i_4)a_I\frac{\p}{\p a_I}=\sum_{I\in\mathcal{E}} (i_1+i_2-i_3-i_4)a_I\frac{\p}{\p b_{\Phi(I)}}.
\]
Note that if $\Phi(I)=\Phi(I')$, then we must have $i_5=i_5'$ and
\[
\sum_{k=1}^4(i_k-i_k')v_k=0.
\]
The only relation among $v_1,v_2,v_3,v_4$ is $v_1-v_2-v_3+v_4=0$, so we must have $i_1=i_1'+l$, $i_2=i_2'-l$, $i_3=i_3'-l$, and $i_4=i_4'+l$ for some $l\in\mathbb{Z}$. This also implies $i_0=i_0'$. Thus,
\[
i_1+i_2-i_3-i_4=(i_1'+l)+(i_2'-l)-(i_3'-l)-(i_4'+l)=i_1'+i_2'-i_3'-i_4'.
\]
Hence under the transformation $\Phi^t$, for $J\in\Delta\cap M$,
\begin{align*}
\sum_{I\in\mathcal{E},\Phi(I)=J} (i_1+i_2-i_3-i_4)a_I\frac{\p}{\p b_{\Phi(I)}}&=(i_1+i_2-i_3-i_4)\sum_{I\in\mathcal{E},\Phi(I)=J} a_I\frac{\p}{\p b_{\Phi(I)}}\\&= (-j_1+2j_3-j_4) b_J\frac{\p}{\p b_J}
\end{align*}
and finally we get 
$$\sum_{I\in\mathcal{E}}(i_1+i_2-i_3-i_4)a_I\frac{\p}{\p a_I}\\
=\sum_{J\in\Delta\cap M} (-j_1+2j_3-j_4) b_J\frac{\p}{\p b_J}.
$$
For other two generators, we may check them similarly.

\subsubsection{ Off diagonal actions (A)} Suppose that
\[
x=\begin{pmatrix}0 & 1 & 0 & 0\\
0 & 0 & 0 & 0 \\
0 & 0 & 0 & 0\\
0 & 0 & 0 & 0
\end{pmatrix}\in \fsl_4.
\]
The induced action of $\exp(xs)$ on the coordinates $z_{ij}$ is given by
\[
z_{12}\mapsto z_{12},~z_{13}\mapsto  z_{13},~z_{14}\mapsto z_{14}, z_{23}\mapsto z_{23}+sz_{13},~z_{24}\mapsto z_{24}+sz_{14},~z_{34}\mapsto z_{34}
\]
and thus the action on the monomial is given by
\[
z_{12}^{i_0}z_{13}^{i_1}z_{14}^{i_2}z_{23}^{i_3}z_{24}^{i_4}z_{34}^{i_5}\mapsto z_{12}^{i_0}z_{13}^{i_1}z_{14}^{i_2}(z_{23}+sz_{13})^{i_3}(z_{24}+sz_{14})^{i_4}z_{34}^{i_5}.
\]
To compute the Lie algebra homomorphism, it suffices to keep the coefficient of $s$, that is, we only need to consider the homomorphism
\[
z_{12}^{i_0}z_{13}^{i_1}z_{14}^{i_2}z_{23}^{i_3}z_{24}^{i_4}z_{34}^{i_5}\mapsto i_3 z_{12}^{i_0}z_{13}^{i_1+1}z_{14}^{i_2}z_{23}^{i_3-1}z_{24}^{i_4}z_{34}^{i_5}+i_4 z_{12}^{i_0}z_{13}^{i_1}z_{14}^{i_2+1}z_{23}^{i_3}z_{24}^{i_4-1}z_{34}^{i_5}.
\]
To simplify the notation, we write
\[
z^I \mapsto i_3 z^{I+(0,1,0,-1,0,0)}+i_4z^{I+(0,0,1,0,-1,0)}
\]
and thus the corresponding operator $Z(x)$ is given by
\[
\sum_{I\in\mathcal{E}} a_I \left( i_3~\frac{\p}{\p a_{I+(0,1,0,-1,0,0)}}+i_4~\frac{\p}{\p a_{I+(0,0,1,0,-1,0)}}\right)=\sum_{I\in\mathcal{E}} (i_3+i_4)a_I\frac{\p}{\p b_{\Phi(I)+\alpha}},
\]
where $\alpha=(0,0,1,0)$. Here, just as before, if $\Phi(I)=\Phi(K)$, then $i_3+i_4=k_3+k_4$. Since $\alpha$ satisfies $\langle \alpha,\delta(e_4)\rangle=-1$, $\langle \alpha,\delta(e_2)\rangle=1$, and $\langle \alpha,\delta(e_i)\rangle=0$ for other $i$, it is a root with $\rho_\alpha=\delta(e_4)=(0,0,-1,1)$ and the corresponding differential operator in the extended GKZ system is
\[
Z_\alpha=\sum_{J\in\Delta\cap M} \langle \rho_\alpha, J-\alpha \rangle b_J \frac{\p}{\p b_{J+\alpha}}=\sum_{J\in\Delta\cap M}  (1-j_3+j_4) b_J \frac{\p}{\p b_{J+\alpha}}.
\]
Also, $1-j_3+j_4=1-(2-i_3-i_4-i_5)+(1-i_5)=i_3+i_4$. Hence the operator $Z_\alpha$ in the extended GKZ system on $Y$ is equal to the operator $Z(x)$ in the tautological system on $X$ under the transformation $\Phi^t$.

We remark that the action of $\exp(xs)$ on $X$ degenerates to the action on $\overline{X}$ by the same formula. In terms of coordinate $w_i$'s, it is given by
\[
w_4\mapsto w_4+sw_2,~~w_i\mapsto w_i~~\text{for other $i$},
\]
which corresponds to the root $(0,0,1,0)$.

\subsubsection{ Off diagonal actions (B)} Suppose that
\[
x=\begin{pmatrix}0 & 0 & 0 & 1\\
0 & 0 & 0 & 0 \\
0 & 0 & 0 & 0\\
0 & 0 & 0 & 0
\end{pmatrix}\in \fsl_4.
\]
 It's easy to compute the induced action of $\exp(xs)$ on $z_{ij}$ is given by
\[
z_{12}\mapsto z_{12},~z_{13}\mapsto  z_{13},~z_{14}\mapsto z_{14}, z_{23}\mapsto z_{23},~z_{24}\mapsto z_{24}-sz_{12}t,~z_{34}\mapsto z_{34}-sz_{13}.
\]
Here $t$ is the degeneration coordinate from $G(2,4)$ to $P(2,4)$.
The key point is that in order to make the $SL_4$-action to be invariant on each fiber, we have to adjust the action by the factor $t$. On way to do this is to fix an isomorphism $X_t\cong X_1=X$ and then use the induced $SL_4$-action on $X_t$. 

We compute the Lie algebra homomorphism as before and get
\[
z_{12}^{i_0}z_{13}^{i_1}z_{14}^{i_2}z_{23}^{i_3}z_{24}^{i_4}z_{34}^{i_5}\mapsto -ti_4 z_{12}^{i_0+1}z_{13}^{i_1}z_{14}^{i_2}z_{23}^{i_3}z_{24}^{i_4-1}z_{34}^{i_5}-i_5 z_{12}^{i_0}z_{13}^{i_1+1}z_{14}^{i_2}z_{23}^{i_3}z_{24}^{i_4}z_{34}^{i_5-1}.
\]
Again we write
\[
z_I \mapsto -ti_4 z_{I+(1,0,0,0,-1,0)}-i_5z_{I+(0,1,0,0,0,-1)}.
\]
and the corresponding operator is given by
\[
\sum_{I\in\mathcal{E}} a_I \left( -ti_4~\frac{\p}{\p a_{I+(1,0,0,0,-1,0)}}-i_5~\frac{\p}{\p a_{I+(0,1,0,0,0,-1)}}\right).
\]
The operator degenerates to 
\[
-\sum_{I\in\mathcal{E}} a_I i_5~\frac{\p}{\p a_{I+(0,1,0,0,0,-1)}}
\]
as $t\mapsto 0$. Note that $\Phi(I+(0,1,0,0,0,-1))=\Phi(I)+(0,1,1,1)$. Also, if $\Phi(I)=\Phi(K)$, then $i_5=k_5$. So the operator becomes (under $\Phi^t$)
\[
-\sum_{I\in\mathcal{E}} i_5a_I\frac{\p}{\p b_{\Phi(I)+\alpha}}, \quad \alpha=(0,1,1,1).
\]
Since $\alpha$ satisfies $\langle \alpha,\delta(e_6)\rangle=-1$, $\langle \alpha,\delta(e_2)\rangle=\langle \alpha,\delta(e_3)\rangle=1$, and $\langle \alpha,\delta(e_i)\rangle=0$ for other $i$, it is also a root with $\rho_\alpha=\delta(e_6)=(0,0,0,-1)$ and the corresponding operator in the extended GKZ system on $Y$ is
\begin{align*}
Z_\alpha&=\sum_{J\in\Delta\cap M} \langle \rho_\alpha, J-\alpha \rangle b_J \frac{\p}{\p b_{J+\alpha}}\\
&=\sum_{J\in\Delta\cap M}  \langle (0,0,0,-1), (j_1,j_2-1,j_3-1,j_4-1) \rangle b_J \frac{\p}{\p b_{J+\alpha}}.
\end{align*}
Note that  $1-j_4=i_5$ and thus the operator $Z_{\alpha}$ is equivalent to the operator $Z(x)$ under the transformation $\Phi^t$ and $t\mapsto 0$.

Again, we remark that 
as $t\to 0$, the action becomes
\[
z_{12}\mapsto z_{12},~z_{13}\mapsto  z_{13},~z_{14}\mapsto z_{14}, z_{23}\mapsto z_{23},~z_{24}\mapsto z_{24},~z_{34}\mapsto z_{34}-sz_{13}.
\]
 and in terms of the coordinate $w_i$'s, it is given by
\[
w_6\mapsto w_6-sw_2w_3,~~w_i\mapsto w_i~~\text{for other $i$},
\]
which corresponds to the root $(0,1,1,1)$.

\subsubsection{Table for equivalence}\label{table}  We can check all generators $x\in \fg=\fsl(4,\mathbb{C})$ similarly and get the correspondence between the symmetry operators in the tautological system on $X$ and the operators in the extended GKZ system on $Y$ as follows.

\begin{center}
\begin{tabular}{|c|l|l|}
\hline  
  $\fsl_4$-actions & Corresponding Torus actions/Root actions \\\hline  
  $E_{11}-E_{22}$   & Torus action $(-1,0,2,-1)\in M$    \\\hline
  $E_{11}-E_{33}$ & Torus action $(1,-1,1,1)\in M$    \\\hline
  $E_{11}-E_{44}$ & Torus action $(0,1,1,0)\in M$    \\\hline
  $E_{12}$ & Root action $(0,0,1,0)$     \\\hline
  $E_{13}$ & Root action $(0,0,1,1)$    \\\hline
  $E_{14}$ & Root action $(0,1,1,1)$   \\\hline
  $E_{23}$ & Root action $(0,0,0,1)$    \\\hline
  $E_{24}$ & Root action $(0,1,0,1)$   \\\hline
  $E_{34}$ & Root action $(0,1,0,0)$   \\\hline
  $E_{21}$ & Root action $(0,0,-1,0)$     \\\hline
  $E_{31}$ & Root action $(-1,0,-1,0)$     \\\hline
  $E_{41}$ & Root action $(-1,-1,-1,0)$     \\\hline
  $E_{32}$ & Root action $(-1,0,0,0)$     \\\hline
  $E_{42}$ & Root action $(-1,-1,0,0)$     \\\hline
  $E_{43}$ & Root action $(0,-1,0,0)$     \\\hline
\end{tabular}\\
\end{center}

\begin{remark}
For $Y$, there are 14 roots:
\begin{align*}
&(-1,0,0,0),~~(-1-1,0,0),~~(-1,0,-1,0),~~(-1-1-1,0),~~(-1,-1,-1,-1)\\
&(0,-1,0,0),~~(0,0,-1,0),~~(0,0,1,0),~~(0,1,0,0)\\
&(0,1,1,1),~~(0,1,0,1),~~(0,0,1,1),~~(0,0,0,1),~~(1,1,1,1)
\end{align*}
and among them, the roots $(1,1,1,1)$ and $(-1,-1,-1,-1)$ are not in the table. 
\end{remark}

\subsection{Characterization of the missing roots}

\begin{lem}
The root automorphisms of $Y$ corresponding to the roots $\alpha_1:=(1,1,1,1)$ and $\alpha_2:=(-1,-1,-1,-1)$ are characterized by those root automorphisms which do not leave the fibers of $\psi:Y\to\overline{X}$ over the maximal singular cone invariant.
\end{lem}
\begin{proof}
Recall that the quotient construction of $Y$  is 
$
\mathbb{C}^6 -Z(\Sigma')/H  
$
where $H$ is the character group of the cokernel of $M\hookrightarrow \mathbb{Z}^{\Sigma(1)}$ and
\[
Z(\Sigma')= \mbox{the zero locus of}\ \langle w_1w_3,w_1w_5,w_2w_3, w_2w_5, w_3w_4, w_4w_5, w_6w_3,w_6w_5\rangle.
\]
The root $(1,1,1,1)$ corresponds to the one-parameter automorphism (on $S$)
\[
w_6\mapsto w_6-\lambda w_1,~~~w_i\mapsto w_i,~~i\neq 6.
\]
Since $Y$ is a geometric quotient, the pre-image of the exceptional locus in $\bbC^6-Z(\Sigma')$ is given by the equation $w_2=w_4=0$ and the set given by $w_2=w_4=w_6=0$ is exactly the fiber over the point $\{\delta(e_2),\delta(e_3),\delta(e_4),\delta(e_5),\delta(e_6)\}$ in $\overline{X}$. This locus is moved by the action corresponding to $(1,1,1,1)$.

One can check that the automorphism given by $(-1,-1,-1,-1)$ does not fix the locus $w_1=w_2=w_4=0$ and all the other automorphisms leave these fibers invariant.
\end{proof}

In fact, $\alpha_1$, $\alpha_2$ are also characterized by ``moving" the singular locus of $\overline{X}$. Hence we can restate the lemma as follows.
\begin{cor}\label{root}
The root automorphisms of $Y$ corresponding to the roots $\alpha_1:=(1,1,1,1)$ and $\alpha_2:=(-1,-1,-1,-1)$ are characterized by those root automorphisms which do not leave the fibers of $\psi$ invariant.
\end{cor}

\begin{remark}
For $\alpha\in R(\Sigma,N)-\{\alpha_1,\alpha_2\}$, let $g_\alpha$ be the corresponding (one paramter) automorphisms on $Y$ and $Z$ be the exceptional locus of the resolution $Y\to \overline{X}$. For a smooth section $f_0$ in $H^0(Y, -K_Y)$ and a cycle $\gamma\in H_3(Y_{f_0},Z\cap Y_{f_0})$. Let $\tau(\gamma)$ be a small topological $S^1$-bundle of $\gamma$ around $Y_{f_0}$.  The {\it relative periods}
\[
\int_{\gamma}\Res\frac{\Omega}{f}=\int_{\tau(\gamma)}\frac{\Omega}{f}
\]
is $g_\alpha$-invariant. Here $f\in H^0(Y,-K_Y)$. As we point out earlier,
\[
Z_\alpha \int_{\gamma}\Res\frac{\Omega}{f}=Z_\alpha\int_{\tau(\gamma)}\frac{\Omega}{f}=\int_{\tau(\gamma)}Z_\alpha\frac{\Omega}{f}
\]
since $g_\alpha$ fixes $Z$. 
\end{remark}

\section{Tautological systems on smooth toric varieties}
We try to propose a construction of the tautological system on a smooth toric variety with the base-point free anti-canonical divisor such that it is possible to study the B-model for Calabi-Yau {\it complete intersections} in the toric variety $\hat{P}(k, n)$. In this section, $Y$ will be a projective smooth toric variety defined by a fan $\Sigma$ with $-K_Y$ being base-point free (or big and nef).

\subsection{Tautological system $\tau_Y(\aut^0(Y),-K_Y,\beta)$}\label{natural TS}
Let $G=\aut^0(Y)$ and $\beta=(0;\id)$.  The $G$-equivariant base-point free anti-canonical line bundle $-K_Y\to Y$ gives a natural $G$-representation $V:=H^0(Y,-K_Y)^*$ and thus gives rise to a $G$-equivariant map $Y\rightarrow \mathbb{P}V$. Let $\bbC^\cross$ act on $V$ by scaling and then $V$ becomes a $\hat{G}(:=G\times\bbC^\cross)$-module. Denote by $Z:\hat{\mathfrak{g}}\to\End(V)$ the corresponding Lie algebra homomorphism. Let $\hat{Y}$ be the cone over the image under the map and $I:=I(\hat{Y},V)\subset\bbC[V]$ be its homogeneous ideal. The tautological system is defined to be the quotient $D$-module
\[
\tau_Y(G,-K_Y,\beta):=D_{V^*}/D_{V^*}\langle p(\p_\zeta):p(\zeta)\in I\rangle+D_{V^*}\langle Z(x)+\beta(x):x\in \hat{\fg}\rangle.
\]

\begin{prop}\label{1st-iso}
The tautological system $\tau_Y(G,-K_Y,\beta)$ is equivalent to the extended GKZ system on $Y$.
\end{prop}
\begin{proof}
Recall that $H^0(Y, -K_Y)$ has a natural basis given by $t^\mu$ in the torus coordiantes with $\mu\in \Delta_{-K_Y}$ and $g\in G$ acts on $\sigma\in H^0(Y, -K_Y)$ by $(g\cdot\sigma)(x):=g(\sigma(g^{-1}x))$, $x\in Y$.  Hence, for $s=(s_1,\cdots, s_n)\in T_N\subset G$, 
\[
s\cdot t^\mu=s^{-\mu}t^\mu
\]
and thus the infinitesimal version of $T_N$-action gives the PDEs
\[
-\sum_\mu \langle \mu,e_j\rangle b_\mu\frac{\p}{\p b_\mu},~j=1,\cdots,n,
\]
which are the operators of order one in the GKZ system on $Y$. 

Next, let $\alpha\in R(\Sigma, N)$ be a root and $\lambda\in \mathbb{C}$. For convenience, we put $\eta_i=t_i(1+\lambda t^\alpha)^{\langle e_i,\rho_\alpha\rangle}$ and use $t_i(\p/\p t_i)$ as our $T_N$-invariant frame. 
For $\lambda$ sufficiently small, one computes
\begin{align*}
t_i\frac{\p}{\p t_i}&=\sum_j t_i\frac{\p \eta_j}{\p t_i}\frac{\p}{\p \eta_j}\\
&=\eta_i\frac{\p}{\p\eta_i}+\frac{\lambda\alpha_i\langle e_i,\rho_\alpha\rangle t^\alpha}{1+\lambda t^\alpha}\eta_i\frac{\p}{\p \eta_i}+\sum_{j\neq i}\frac{\lambda\alpha_i\langle e_j,\rho_\alpha\rangle t^\alpha}{1+\lambda t^\alpha}\eta_j\frac{\p}{\p \eta_j}
\end{align*}
Then, we obtain
\[
(\star_\lambda)\bigwedge_{i=1}^n \eta_i\frac{\p}{\p \eta_i} = \bigwedge_{i=1}^n t_i\frac{\p}{\p t_i},
\]
where 
\[
(\star_\lambda)=1+\sum_{i=1}^n\frac{\lambda\alpha_i\langle e_i,\rho_\alpha\rangle t^\alpha}{1+\lambda t^\alpha}+\text{higher order terms in}\ \lambda.
\]
For a section $t^\mu$, $g_{\alpha,\lambda}\cdot t^\mu = (\star_\lambda)t^\mu (1+\lambda t^\alpha)^{-\langle \mu,\rho_\alpha\rangle}$, so the infinitesimal version is
\[
t^\mu\mapsto \sum_{i=1}^n \alpha_i\langle e_i,\rho_\alpha\rangle t^{\alpha}t^\mu-\langle \mu,\rho_\alpha\rangle t^\alpha t^\mu=-\langle\mu-\alpha,\rho_\alpha\rangle t^{\mu+\alpha}.
\]
and thus gives us the operator
\[
-\sum_\mu\langle\mu-\alpha,\rho_\alpha\rangle b_\mu\frac{\p}{\p b_{\mu+\alpha}},
\]
which coincides with $Z_\alpha$ in the extended GKZ system on $Y$.

For the polynomial operators in $\tau_Y(\aut^0(Y),-K_Y,\beta)$, we recall that the ideal of the image of $Y$ under the morphism determined by $-K_Y$ is generated by its lattice relations.  More precisely, put $\sA=\{(\mu,1):\mu\in \Delta_{-K_Y}\}$ and $q=|\sA|$. Via $e_i\mapsto (\mu_i,1)\in \sA\subset M\times\bbZ$, it gives an exact sequence $0\to L\to \mathbb{Z}^q\to M\times\bbZ$.  For $l\in L$, we have $\sum l_i=0$ and $\sum l_i \mu_i=0$. We write $l=l^+-l^-$ with $l^\pm\in(\mathbb{N}\cup\{0\})^q$. The (homogeneous) ideal of the image of $Y$ under the morphism determined by $-K_Y$ is generated by $z^{l^+}-z^{l^-}$. Hence, the box operators in the extended GKZ system and the polynomial operators in the tautological system are the same.

Finally, note that the Euler operator comes from the scaling action.  Hence the proof is complete.
\end{proof}

\subsection{Tautological system $\tau_Y(\aut_g(S), -K_Y, \beta_\chi)$}\label{general TS} 
For the flexibility (since in general $\sL$ is not $\aut^0(Y)$-linearizable, e.g., $\sO(1)\to\mathbb{P}^n$ is not $PGL_n$-linearizable), 
we consider $G=\aut_g(S)$ and $V=H^0(Y, -K_Y)^*$. Recall that for the smooth projective toric variety $Y$, if $D=\sum_{\rho\in \Sigma(1)} a_\rho D_\rho$ is a $T_N$-invariant Weil divisor with $d=[D]\in \cl(Y)$, then the map 
$$
H^0(Y,\sO_Y(D))\to S_d,~~ t^\mu\mapsto \prod_\rho w_\rho^{\langle \mu,\rho\rangle+a_\rho}
$$ 
establishes an isomorphism. Hence, instead of $\aut^0(Y)$, there is a natural $G$-action on $H^0(Y,\sO_Y(D))$, in particular for $D=-K_Y$. Let $r:=|\Sigma(1)|$. We can define an equivalence relation on $\Sigma(1)$: 
\[
\rho\sim\rho' \Leftrightarrow \deg(w_\rho)=\deg(w_{\rho'})\in A_{n-1}(Y).
\]
It gives a partition $\Sigma(1)=\Sigma_1\cup\cdots\cup\Sigma_h$, where $\Sigma_i$ consists of the variables of degree $d_i$. Write $S_i=S_{d_i}=S_{d_i}'\oplus S_{d_i}''$, where $S_{d_i}'$ is the space spanned by $w_\rho$ for $\rho\in\Sigma_i$ and $S_{d_i}''$ is spanned by the remaining monomials. Recall that we have a (split) exact sequence of groups (Equation ({\sc{e7}}), \cite{C2014}):
\[
0\to 1+\mathcal{N}\to\aut_g(S)\to \prod_{i=1}^h \gl(S_i')\to 0.
\]
Let $\chi:\aut_g(S)\to\bbC^\cross$ be the pull-back of the product of determinant functions on $\gl(S_i')$, which, when restricts on the maximal torus $(\bbC^\cross)^r\subset\aut_g(S)$, is given by
\[
\left.\chi\right|_{(\bbC^\cross)^r}:(\lambda_1,\cdots,\lambda_r)\mapsto \lambda_1\cdots\lambda_r.
\]
 Let $\beta_\chi:\hat{\fg}\to\bbC$ be a Lie algebra homomorphism given by $(d\chi;\id)$ on $\fg\times\bbC$. 
\begin{thm}\label{general on Y}
The tautological system $\tau_Y(G, -K_Y, \beta_\chi)$ is equivalent to the extended GKZ system on $Y$.
\end{thm}
\begin{proof}
For a root $\alpha\in R(\Sigma, N)$, the induced one-parameter automorphism on $V^*$ is given by 
\[
w_{\rho_\alpha}\mapsto w_{\rho_\alpha}+\lambda\prod_{\rho\neq \rho_\alpha} w_{\rho}^{\langle \alpha,\rho\rangle},~~ w_{\rho}\mapsto w_{\rho},~\text{for $\rho\neq \rho_\alpha$}
\]
and its infinitesimal version on basis is given by
\begin{align*}
\prod_\rho w_\rho^{\langle \mu,\rho\rangle+1} \mapsto \langle \mu-\alpha,\rho_\alpha\rangle \prod_\rho w_\rho^{\langle \mu+\alpha,\rho\rangle+1}.
\end{align*}
In terms of matrices, it can be represented by $(t^\mu)\mapsto B(t^\mu)$ with $B_{\mu+\alpha,\mu}=\langle \mu-\alpha,\rho_\alpha\rangle$ and $B_{ij}=0$ for other entries. Taking the dual representation, we have $(b_\mu) = B^t (b_\mu)$. Hence this gives the operator 
\[
\sum_\mu \langle \mu-\alpha,\rho_\alpha\rangle b_\mu\frac{\p}{\p b_{\mu+\alpha}}
\]
as expected. Note that if $E_\alpha\subset \fg$ is the corresponding root subspace, then $\beta_\chi(E_\alpha)=0$ by construction. Hence $Z(E_\alpha)+\beta_\chi(E_\alpha)=Z_\alpha$.

For the torus actions, put $x_1=(1,0,\ldots,0)\in \bbC^r$. The one-parameter subgroup gives $w_1\mapsto e^{\lambda} w_1$ and $w_i\mapsto w_i$ for $i\neq 1$. It induces
\[
\prod_\rho w_\rho^{\langle \mu,\rho\rangle+1} \mapsto (e^{\lambda})^{\langle \mu,\rho_1\rangle+1}\prod_\rho w_\rho^{\langle \mu,\rho\rangle+1}.
\]
and the corresponding infinitesimal version is
\[
\prod_\rho w_\rho^{\langle \mu,\rho\rangle+1} \mapsto (\langle \mu,\rho_1\rangle+1)\prod_\rho w_\rho^{\langle \mu,\rho\rangle+1}.
\]
It gives the operator
\[
\sum_\mu (\langle \mu,\rho_1\rangle+1) b_\mu\frac{\p}{\p b_\mu}.
\]
Since $\beta_\chi(x_1)=1$ by construction, 
\[
Z(x_1)+\beta_\chi(x_1)=\sum_\mu (\langle \mu,\rho_1\rangle+1) b_\mu\frac{\p}{\p b_\mu}+1=\sum_\mu \langle \mu,\rho_1\rangle b_\mu\frac{\p}{\p b_\mu}+\sum_\mu b_\mu\frac{\p}{\p b_\mu}+1.
\]
Hence together with the Euler operator, $Z(x_1)+\beta(x_1)$ generates the operator
\[
\sum_\mu \langle \mu,\rho_1\rangle b_\mu\frac{\p}{\p b_\mu}.
\]
Similarly, we can produce 
\[
\sum_\mu \langle \mu,\rho_i\rangle b_\mu\frac{\p}{\p b_\mu},~i=1,\cdots,r.
\]
Since $\rho_i$'s span the full lattice $N$, these can generate all operators
\[
\sum_\mu \langle \mu,e_j\rangle b_\mu\frac{\p}{\p b_\mu},~j=1,\cdots,n,
\]
which are the operators of order one (except for the Euler operator) in the GKZ system.

The polynomial operators can be checked as in Proposition \ref{1st-iso}, so we complete the proof.
\end{proof}
\begin{remark}\label{Euler}
For the case $Y=\hat{P}(2,4)$, since $\rho_i$'s have relations, namely $\rho_1+\rho_2+\rho_4+\rho_6=0$, the auxiliary action of $\bbC^\cross$ is redundant. Indeed, we take $(\lambda,\lambda,1,\lambda,1,\lambda)\in(\bbC^\cross)^6$ and
\[
\sum_{i=1,2,4,6}\sum_\mu (\langle \mu,\rho_i\rangle+1) b_\mu\frac{\p}{\p b_\mu}=4\sum_\mu b_\mu\frac{\p}{\p b_\mu}.
\]
Also, $\beta((1,1,0,1,0,1);0)=4$. This action already generates the Euler operator.
\end{remark}

\subsection{Period integrals of CY complete intersections in $\hat{P}(k,n)$}\label{TS-ci}
Via the toric degeneration, we will propose a way to study the B-model of a Calabi-Yau complete intersection in $\hat{P}(k,n)$ (See \cite{BCKS2000} for the general construction of $\hat{P}(k,n)$.) 

Tautological systems were constructed to govern the period integrals of Calabi-Yau complete intersections in a Grassmannian $X=G(k, n)$ (more generally, a partial flag variety). Suppose that $-K_X$ admits a decomposition $-K_X=L_1+\cdots+L_s$ such that $L_i$'s are $G (=SL_n)$-linearized base-point free line bundles. Let $V_i=H^0(X,L_i)^*$. We have a $G$-map $X\to\mathbb{P}V_1\times\cdots\times\mathbb{P}V_s$. Let $\hat{X}$ be the cone over the image in $V:=V_1\times\cdots\times V_s$ and $I:=I(\hat{X},V)$. Let $G\to\aut(V)$ be the associated representation. We extend this representation to $\hat{G}:=G\times(\bbC^\cross)^s\to\aut(V)$ by letting $i$-th $\bbC^\cross$ act on $V_i$ via scaling. Let $Z:\hat{\fg}=\fg\times\bbC^s\to \End(V)$ be the corresponding Lie algebra homomorphism and $\beta:\hat{\fg}\to\bbC$ be a Lie algebra homomorphism. The general tautological system is defined to be the quotient $D$-module
\begin{align}\label{ctau}
\tau_X(G,L_1,\cdots,L_s,\beta):=D_{V^*}/D_{V^*}\langle p(\p_\zeta):p(\zeta)\in I\rangle+D_{V^*}\langle Z(x)+\beta(x):x\in \hat{\fg}\rangle,
\end{align}
which will govern the period integrals of the Calabi-Yau complete intersections in $X$ defined by $L_i$'s. Note that $X$ has Picard number one and $K_X\isom \sO_X(-n)$, so we can write $L_i=\sO_X(n_i)$ with $\sum_{i=1}^s n_i=n$.

To get an analogue for $Y=\hat{P}(k,n)$, let $G=\aut_g(S)$ and $r=|\Sigma(1)|$. By Proposition \ref{linearization}, $H^0(Y,\sL)$ is a $G$-space. Let $\sL_i=n_i\sL$. By the fact that $-K_Y=n\sL$, we have
\[
-K_Y=\sL_1+\cdots+\sL_s.
\]
Also, $H^0(Y,\sL_i)$ is a $G$-space (It is the reason why we have to pick the group $\aut_g(S)$ in Section \ref{general TS} instead of  $\aut^0(Y)$.) and we can choose integers $a_{ij}$ so that $\sL_i\sim \sum_{j} a_{ij} D_j$ and $\sum_{i} a_{ij}=1$ for all $j$. Based on these choices, the natural multiplication map
\[
H^0(Y, \sL_1)\times\cdots\times H^0(Y, \sL_s)\to H^0(Y, -K_Y)
\]
is just the multiplication in $S$ and is $G$-equivariant when we identify $H^0(Y, \sL_i)$ with its corresponding graded piece of $S$.

Put $\hat{G}=G\times(\bbC^\cross)^s$ and let $i$-th $\bbC^\cross$ act on $H^0(Y, \sL_i)$ via scaling. Let $\beta_\chi=(d\chi;\id,\cdots,\id):\hat{\fg}\to\bbC$ be given as in Section \ref{general TS}. Now we can get the tautological system $\tau_Y(G,\sL_1,\cdots,\sL_s,\beta_\chi)$ as above.
\begin{prop}
$\tau_Y(G,\sL_1,\cdots,\sL_s,\beta_\chi)$ governs the period integrals of the Calabi-Yau complete intersections in $Y$ cut down by $H^0(Y, \sL_i)$.
\end{prop}
\begin{proof}
Here, $V_i=H^0(Y, \sL_i)^*$ and $V=V_1\times\cdots\times V_s$. The period integral of the complete intersections is given by
\[
\Pi_\gamma=\int_{\gamma} \Res\frac{1}{\sigma_1\cdots \sigma_s}\bigwedge_{i=1}^N \frac{dt_i}{t_i},~N=\dim\hat{P}(k,n)
\]
and $\sigma_i\in H^0(Y, \sL_i)$. In terms of torus coordinates, we write $\sigma_i=\sum_{\mu_i\in\Delta_{\sL_i}} b_{\mu_i} t^{\mu_i}$. Starting from the torus actions, we consider the action $g_j(\lambda)$: $t_j\mapsto \lambda t_j$, $t_i\mapsto t_i$ for $i\neq j$ and associate it with a period integral:
\[
\Pi_\gamma= \int_{g_j(\lambda)_*\tau(\gamma)} g_j(\lambda)^* \left(\frac{1}{\sigma_1\cdots \sigma_s}\right)\bigwedge_{i=1}^N \frac{dt_i}{t_i}=\int_{\tau(\gamma)} g_j(\lambda)^* \left(\frac{1}{\sigma_1\cdots \sigma_s}\right)\bigwedge_{i=1}^N \frac{dt_i}{t_i},
\]
where the later equality holds if $\lambda$ is close to $1$. Now applying the operator $\lambda(\p /\p \lambda)$ and taking $\lambda=1$, we have 
\[
0=\int_\gamma \frac{(-\sum_i \mu_{ij}\sigma_i\prod_{k\neq i}\sigma_k)}{(\sigma_1\cdots \sigma_s)^2}\bigwedge_{i=1}^N \frac{dt_i}{t_i}=\left(\sum_{i,\mu_i\in\Delta_{\sL_i}} \langle \mu_i,e_j\rangle b_{\mu_i}\frac{\p}{\p b_{\mu_i}}\right) \Pi_\gamma.
\]
Clearly, they should also be annihilated by the Euler operators
\[
\left(\sum_{\mu_i\in\Delta_{\sL_i}} b_{\mu_i} \frac{\p}{\p b_{\mu_i}} +1\right)\Pi_\gamma =0,~i=1,\cdots,s.
\]

Let $x_1\in\bbC^r\subset \fg$, say $x_1=(1,0,\cdots,0)$. The homogenization of $\sigma_i$ is given by $\sum_{\mu_i} b_{\mu_i} \prod_{j=1}^r w_j^{\langle\mu_i,\rho_j \rangle+a_{ij}}$. Similar to the proof in Theorem \ref{general on Y}, we have
\[
Z(x_1)+\beta(x_1)=\sum_{i,\mu_i\in\Delta_{\sL_i}} \langle\mu_i,\rho_1 \rangle b_{\mu_i} \frac{\p}{\p b_{\mu_i}} + \sum_{i,\mu_i\in\Delta_{\sL_i}} a_{i1} b_{\mu_i} \frac{\p}{\p b_{\mu_i}}+1
\]
and thus
\[
(Z(x_1)+\beta(x_1))\Pi_\gamma = \left(\sum_{i,\mu_i\in\Delta_{\sL_i}} \langle\mu_i,\rho_1 \rangle b_{\mu_i} \frac{\p}{\p b_{\mu_i}}\right)\Pi_\gamma + \left(\sum_{i,\mu_i\in\Delta_{\sL_i}} a_{i1} b_{\mu_i} \frac{\p}{\p b_{\mu_i}}+1\right)\Pi_\gamma=0,
\]
since $\sum_{i} a_{i1}=1$ by choice. Also, 
$
(Z(x_j)+\beta(x_j))\Pi_\gamma=0,~j=1,\cdots,r.
$

For a root $\alpha$, we just notice that the one-parameter subgroup, denoted by $g_\alpha(\lambda)$, is given by $t^\mu\mapsto t^\mu(1+\lambda t^\alpha)^{\langle \mu,\rho_\alpha\rangle}$ and thus
$
\sigma_i\mapsto \sum_{\mu_i\in\Delta_{\sL_i}} b_{\mu_i} t^{\mu_i}(1+\lambda t^\alpha)^{\langle\mu_i\rho_\alpha\rangle}.
$
Let
\[
dm_\lambda:=\bigwedge_{i=1}^N \left(\frac{dt_i}{t_i}+\lambda\langle e_i,\rho_\alpha\rangle\frac{1}{1+\lambda t^\alpha}\left(\sum_{j=1}^N \alpha_j t^\alpha\cdot\frac{dt_j}{t_j}\right)\right).
\]
Then 
\[
\left.\frac{\p}{\p\lambda}dm_\lambda\right|_{\lambda=0}=\sum_{i=1}^N \alpha_i\langle e_i,\rho_\alpha\rangle t^\alpha dm=\langle \alpha,\rho_\alpha\rangle t^\alpha dm,~dm=\bigwedge_{i=1}^N \frac{dt_i}{t_i}.
\]
Also, by the product rule,
\[
\left.\frac{\p}{\p\lambda}g_\alpha(\lambda)^*\left(\frac{1}{\sigma_1\cdots\sigma_s}\right)\right|_{\lambda=0}=\frac{-\sum_i (\sum_{\mu_i\in\Delta_{\sL_i}} b_{\mu_i} \langle\mu_i, \rho_\alpha\rangle t^{\mu_i+\alpha})\prod_{k\neq i}\sigma_k}{(\sigma_1\cdots\sigma_s)^2}.
\]
Suppose $\rho_\alpha=\rho_k$. Similar to (\ref{pf1}), we have
\[
\left(\sum_i\sum_{\mu_i\in\Delta_{\sL_i}} \langle \mu_i-a_{ik}\alpha,\rho_k\rangle b_{\mu_i}\frac{\p}{\p b_{\mu_i+\alpha}}\right)\Pi_\gamma=0.
\]
Note that here we also use the fact $\sum_{i} a_{ik}=1$. Indeed,
\[
\int_{\tau(\gamma)}\left(\frac{-\sum_i (\sum_{\mu_i\in\Delta_{\sL_i}} b_{\mu_i} \langle\mu_i, \rho_k\rangle t^{\mu_i+\alpha})\prod_{k\neq i}\sigma_k}{(\sigma_1\cdots\sigma_s)^2}+\sum_i a_{ik} \frac{\langle\alpha, \rho_k\rangle t^{\alpha}\prod_{k}\sigma_k}{(\sigma_1\cdots\sigma_s)^2}\right)dm=0
\]
The integrand is simplified to
\begin{align*}
&-\frac{\sum_i\left(\sum_{\mu_i} b_{\mu_i}\langle \mu_i,\rho_k\rangle t^{\mu_i+\alpha} \prod_{k\neq i} \sigma_k-a_{ik}(\prod_{k} \sigma_k) \langle\alpha,\rho_k\rangle t^\alpha\right)}{(\sigma_1\cdots\sigma_s)^2}\\
&=-\frac{\sum_i (\prod_{k\neq i}\sigma_k) \left(\sum_{\mu_i} b_{\mu_i}\langle \mu_i,\rho_k\rangle t^{\mu_i+\alpha}-a_{ik}\langle\alpha,\rho_k\rangle t^\alpha\sigma_i\right)}{(\sigma_1\cdots\sigma_s)^2}\\
&=-\frac{\sum_i (\prod_{k\neq i}\sigma_k) \left(\sum_{\mu_i} (b_{\mu_i}\langle \mu_i,\rho_k\rangle t^{\mu_i+\alpha}-a_{ik}b_{\mu_i}\langle\alpha,\rho_k\rangle t^{\alpha+\mu_i})\right)}{(\sigma_1\cdots\sigma_s)^2}\\
&=\left(\sum_i\sum_{\mu_i\in\Delta_{\sL_i}} \langle \mu_i-a_{ik}\alpha,\rho_k\rangle b_{\mu_i}\frac{\p}{\p b_{\mu_i+\alpha}}\right)\Pi_\gamma
\end{align*}
and the operator 
\[
\sum_i\sum_{\mu_i\in\Delta_{\sL_i}} \langle \mu_i-a_{ik}\alpha,\rho_k\rangle b_{\mu_i}\frac{\p}{\p b_{\mu_i+\alpha}}
\]
is exactly the one generated by the root $\alpha$ action on $V$.

It is easy to check that the polynomial operators also kill $\Pi_\gamma$. 
\end{proof}

\section{Tautological system on $G(2,4)$ $\Leftrightarrow$ Extended GKZ system on $\hat{P}(2,4)$}

Both on $G(2,4)$ and $\hat{P}(2,4)$, we have the similar ways to construct their tautological systems, so we are able to study the degeneration of tautological systems under the conifold transition on $G(2,4)$ now.
\subsection{Tautological system on $\hat{P}(2,4)$ for further embedding}
Recall that there is no canonical choice of bases for $H^0(G(k,n), -K_{G(k,n)})$. By using Borel-Weil theorem, we can get a natural injection 
\[
H^0(G(k,n),-K_{G(k,n)})^*\hookrightarrow H^0(\mathbb{P}^{{n \choose k}-1}, \sO(n))^*
\]
and thus we use the later space as the parameter space to construct the tautological system 
via the change of variable formula (\cite{LY2013}) (See Example \ref{TS-g24} for the simplest case $G(2,4)$).

To compare the tautological systems on $X(:=G(2, 4))$ and $Y(:=\hat{P}(2, 4))$, we have to use the good transformation $\Phi^t : H^0(Y, -K_Y)^* \rightarrow H^0(\mathbb{P}^4, \sO(4))^*$, which was introduced in Section \ref{coordinates}.
Note that it falls to be an $\aut^0(Y)$-module homomorphism because $\sL$ is not $\aut^0(Y)$-linearizable. However it is $\aut_g(S)$-equivariant, so we prefer to use $\aut_g(S)$ to construct its tautological system as in Section \ref{general TS}. 

Together with the fact that $\sL$ is $\aut_g(S)$-linearizable, we are able to write down the generators for the tautological system $\tau_Y(\aut_g(S), -K_Y, \beta_\chi)$ on $Y$.
\begin{thm}\label{tauY}
Let $\oZ$ be the composite Lie algebra homomorphism 
$$
\oZ:\lie(\aut_g(S))\mathop{\longrightarrow}^Z {\rm End}(H^0(Y, -K_Y)^*) \longrightarrow {\rm End}(H^0(\mathbb{P}^5,\sO(4)))^*).
$$
The tautological system $\tau_Y(\aut_g(S), -K_Y, \beta_\chi)$ is equivalent to a system generated by the following operators:
\begin{itemize}
\item[1.] $\oZ(x)+\beta_\chi(x)$, $x\in \lie(\aut_g(S))$.
\item[2.] $\p_\zeta$, where $\zeta\in (H^0(Y, -K_Y)^*)^\perp\subset H^0(\mathbb{P}^5,\sO(4))$.
\item[3.] $\p_{\zeta_u}\p_{\zeta_v}-\p_{\zeta_p}\p_{\zeta_q}$, where $u,v,p,q\in\mathcal{E}$ and $u+v=p+q$.
\end{itemize}
Here $\mathcal{E}=\{(i_0,i_1,i_2,i_3,i_4,i_5)\in\mathbb{Z}^6_{\geq 0}:\sum_{k=0}^5 i_k=4\}$
\end{thm}
\begin{proof}
Let $V=H^0(Y,-K_Y)^*$ and $W=H^0(\mathbb{P}^5,\sO(4)))^*$.  Recall the maps 
$$Y \mathop{\longrightarrow}^{\psi_\sL} \mathbb{P}(V) \mathop{\longrightarrow}^{\mathbb{P}(\Phi^t)} \mathbb{P}(W).$$ 
Let $\hat{Y}$ be the cone over the image of $Y$ under the map $\psi_\sL$ and $I_V:=I(\hat{Y}, V) \subset \mathbb{C}[V]$. Also,  let $\tilde{Y}$ be the cone over the image of $Y$ under the map $\mathbb{P}(\Phi^t)\circ\psi_\sL$ and $I_W:=I(\tilde{Y}, W) \subset \mathbb{C}[W]$.
For the given $\beta_\chi$, together with two canonical representations $\aut_g(S)\to \aut(V)$ and $\aut_g(S) \to \aut(W)$, we can construct two associated tautological systems on $Y$. By the change of variable formula in \cite{LY2013}, these two systems are equivalent to each other.

The system $\tau_Y(\aut_g(S), -K_Y, \beta_\chi)$ is equal to the first one and we have to replace the polynomial operators in the second one with the operators in type 2 and 3, that is, 
by the following commutative diagram
\[
\xymatrix{
&Y \ar[r] &\overline{X} \ar@{^{(}->}[r]^{\psi_\sL} \ar@{^{(}->}[dr]_-{\psi_{-K_{\overline{X}}}} & \mathbb{P}^5 \ar@{^{(}->}[r]^{\text{4-Uple}} &\mathbb{P}^{125},\\
&   &   &\mathbb{P}^N \ar@{^{(}->}[ur]_-{\mathbb{P}(\Phi^t)}
}
\]
we have to compare the polynomial operators for $Y$ in $\mathbb{P}^{125}$ with $Y$ in $\mathbb{P}^N$. The map $W\hookrightarrow V$ gives the isomorphism (Note that $Y \to\overline{X}$ is surjective.)
\[
(\bbC[\overline{X}]\cong)~\bbC[V]/I_V\cong\bbC[W]/I_W
\]
and thus the ideal $I_W=I(\overline{X},W)$. By Lemma 9.4 in \cite{LY2013}, $I(\overline{X},W)$ is generated by the Veronese binomials and the space of linear forms vanishing on $\overline{X}$ in $\mathbb{P}(W)$. Hence the polynomial operators with respect to $I_W$ is generated by
\begin{itemize}
\item[1.] $\p_\zeta$, where $\zeta\in W^\perp\subset V^*$.
\item[2.] $\p_{\zeta_u}\p_{\zeta_v}-\p_{\zeta_p}\p_{\zeta_q}$, where $u,v,p,q\in\mathcal{E}$ and $u+v=p+q$.
\end{itemize}
\end{proof}

\subsection{Main theorems}
 We can reconstruct the tautological system on $X$ from the extended GKZ system on $Y$.
\begin{thm}\label{main1}
The tautological system on $X$ degenerates, as a $\sD$-module, to the tautological (sub-)system on $Y$ under the coordinate transformation $\Phi^t$.
\end{thm}
\begin{proof}
Let's recall these two systems.
For $X$, it is generated by
\begin{itemize}
\item[\bf 1.] $Z(x)+\beta(x)$, $x\in \fsl_4\times\mathbb{C}$.
\item[\bf 2.] $\p_\zeta$, where $\zeta\in (H^0(-K_X)^*)^\perp\subset H^0(\mathbb{P}^5,\sO(4))$.
\item[\bf 3.] $\p_{\zeta_u}\p_{\zeta_v}-\p_{\zeta_p}\p_{\zeta_q}$, where $u,v,p,q\in\mathcal{E}$ and $u+v=p+q$.
\end{itemize}
and for $Y$, it is generated by
\begin{itemize}
\item[\bf 1.] $Z(x)+\beta_\chi(x)$, $x\in \lie(\aut_g(S))$.
\item[\bf 2.] $\p_\zeta$, where $\zeta\in (H^0(-K_Y)^*)^\perp\subset H^0(\mathbb{P}^5,\sO(4))$.
\item[\bf 3.] $\p_{\zeta_u}\p_{\zeta_v}-\p_{\zeta_p}\p_{\zeta_q}$, where $u,v,p,q\in\mathcal{E}$ and $u+v=p+q$.
\end{itemize}
The degeneration is given by the coordinates (\ref{degcoor}) and thus, over $t$, the fiber $X_t$ has its own tautological system with the operators of type 2 being given by 
\[
\zeta=(z_{14}z_{23}-z_{13}z_{24}+tz_{12}z_{34})\cdot(\text{degree two polynomials}),
\]
which is clearly equal to the type {\bf 2} operator on $Y$ when $t\to 0$. Type {\bf 3} operators are the same. 

The Euler operator (induced by scaling) in type {\bf 1} for $X$ corresponds to the type {\bf 1} operator given by $H$ for $Y$ (also see Remark \ref{Euler}). Finally, the symmetry operators for the type {\bf 1} have been done in Section \ref{symmetry} and summarized to the table in Section \ref{table}. 
Hence we complete the proof.
\end{proof}

\begin{thm}\label{main2}
The tautological system on $X$ is completely determined by the tautological system on $Y$ (= the extended GKZ system on $Y$). Indeed, it is determined by the open part $Y^\circ:=Y-Z$, where $Z$ is the exceptional locus of the resolution $Y\to\overline{X}$.
\end{thm}
\begin{proof}
We are given the tautological system on $Y$, which is generated by
\begin{itemize}
\item[\bf 1.] $Z(x)+\beta_\chi(x)$, $x\in \lie(\aut_g(S))$.
\item[\bf 2.] $\p_\zeta$, where $\zeta\in (H^0(-K_Y)^*)^\perp\subset H^0(\mathbb{P}^5,\sO(4))$.
\item[\bf 3.] $\p_{\zeta_u}\p_{\zeta_v}-\p_{\zeta_p}\p_{\zeta_q}$, where $u,v,p,q\in\mathcal{E}$ and $u+v=p+q$.
\end{itemize}
For the operators in \textbf{1}, we first remove the operators generated by the roots $(1,1,1,1)$ and $(-1,-1,-1,-1)$, since these operators ``move" the extremal rays in the contraction as we said in Corollary \ref{root}. For the other roots, we are going to extend them to the automorphisms of the general fiber $X_t$. 

If the root $\alpha$ is semi-simple, i.e., $-\alpha$ is also a root, then there is nothing to do since the action can be extended into the general fiber directly. 

Suppose we are given a non semi-simple root automorphism $(0,0,1,1)$. It acts on $S$ via
\[
w_6\mapsto w_6-\lambda w_2w_5,~~ w_i\mapsto w_i,~~i\neq 6.
\]
Using the identification between $z_{ij}$'s and $w_i$'s, this corresponds to 
\[
z_{34}\mapsto z_{34}-\lambda z_{14},~~z_{ij}\mapsto z_{ij},~~\text{for other $i,j$}.
\]
To extend this action to general fiber, we consider the equation of the general fiber
\[
z_{14}z_{23}-z_{13}z_{24}+tz_{12}z_{34}=0.
\]
We can extend this action into the general fiber via
\[
z_{34}\mapsto z_{34}-\lambda z_{14},~~ z_{23}\mapsto z_{23}+t\lambda z_{12},~~ z_{ij}\mapsto z_{ij},~~\text{for other $i,j$}.
\]
Repeat the same process to get all the extensions of the roots other than $(1,1,1,1)$ and $(-1,-1,-1,-1)$.

To recover the torus symmetry on $X$, we consider the abelian part of the Lie bracket $[\fk,\fk]$ of $\fk:=\Lie(\aut_g(S))$, which is a $3$-dimensional vector space over $\mathbb{C}$. It's generated by $(-1,0,2,-1)$, $(1,-1,1,1)$ and $(1,0,0,1)$ in torus coordinates. We remark that these are generated by the Lie brackets of the semi-simple roots in $\fk$.

Hence we obtain $3+4+4+4=15$ symmetries for general fiber. These generate the full symmetry operators for $X$ when we set $t=1$.

For the operators in \textbf{2}, we just put the term $tz_{12}z_{34}$. For example, if $\zeta\in (H^0(-K_Y)^*)^\perp\subset H^0(\mathbb{P}^5,\sO(4))$ is such a linear form, then $\zeta=(z_{14}z_{23}-z_{13}z_{24})\times (\text{degree two terms})$. We just put
\[
\zeta_t=(z_{14}z_{23}-z_{13}z_{24}+tz_{12}z_{34})\times (\text{degree two terms}).
\]
Then $\zeta_1$ gives the type \textbf{2} operator for $X$.

The type \textbf{3} operators are the same and thus there is nothing to do. 
\end{proof}

\begin{remark}
Let $\overline{X}_{f_0}$ be a Calabi-Yau threefold with only four ODPs as before and $U:=\overline{X}_{f_0}-\sing(\overline{X}_{f_0})$. The (infinitesimal) deformations are classified by $H^1(U,\Theta_U)$ \cite{F1989}. The corresponding resolution is $Y_{f_0}$. Put $V:=Y_{f_0}-Z\cap Y_{f_0}$. We have $H^1(U,\Theta_U)\cong H^1(V,\Theta_V)$. Therefore, $R(\Sigma,N)-\{\alpha_1,\alpha_2\}$ should be regarded as an \textit{embedded symmetry} on $V$ or a Gauss-Manin connection on \textit{open varieties}.

On the other hand, the ``dual" of $H^1(V,\Theta_V)$ is $H_3(Y_{f_0},Z\cap Y_{f_0})$ and 
\[
\dim H_3(Y_{f_0},Z\cap Y_{f_0})-\dim H_3(Y_{f_0})=3,
\]
so one would expect the {\it relative periods} are the coordinates of the moduli space around $[\overline{X}_{f_0}]$. This is exactly how we establish the correspondence. Moreover, the three missing operators could be viewed as the ``monodromy operators" in the sense of variation of Hodge structures near the point $[X_{f_0}]$ on the moduli space. The original periods are invariant under the monodromy actions.
\end{remark}

\subsection{Remarks on rigidity}
We would like to show that the Calabi-Yau hypersurfaces in $X$ and $Y$ are deformation complete, i.e., all deformations can be realized as the ones in $X$ and $Y$. Hence we have the full information of $B$ model on hand.

\subsubsection{The moduli of CY hypersurfaces in $Y$}
For a smooth Calabi-Yau hypersurface $Y_f$ in $Y$ with $f\in H^0(Y, -K_Y)$, we claim that the natural map
\begin{eqnarray}\label{map1}
H^0(Y_f, \sN_{Y_f/Y}) \to H^1(Y_f,T_{Y_f})
\end{eqnarray}
is surjective. From the canonical splitting (\cite{B1994},\cite{M2005})
\[
H^1(Y_f,T_{Y_f})=H_{\text{poly}}^1(Y_f,T_{Y_f})\oplus H_{\text{non-poly}}^1(Y_f,T_{Y_f}),
\]
whose polynomial part is identified with the image of (\ref{map1}),  it is sufficient to show that $H_{\text{non-poly}}^1(Y_f,T_{Y_f})=0$. Indeed, the dimension of the non-polynomial part is given by
\[
\dim H_{\text{non-poly}}^1(Y_f,T_{Y_f})=\sum_{\Theta} l^*(\Theta)l^*(\Theta'),
\]
where the summation is taken over all codimenson two faces $\Theta$ in $\Delta:=\Delta_{-K_Y}$. Here, $\Theta'$ is the corresponding dual face in $\Delta'$ which is the dual polytope of $\Delta$ and is equal to the convex hull of $\{\delta(e_1),\cdots,\delta(e_6)\}$ by construction. And, $l^*(\Theta)$ is the number of relative interior lattice points in $\Theta$. 
In our case, $\Theta'$ is a one-dimensional face on $\Delta'$. It contains at least two points in $\{\delta(e_i)\}$, say $\delta(e_k)$, $\delta(e_l)$. Let $H$ be the supporting hyperplane of $\Theta'$. If $\delta(e_k)$ and $\delta(e_l)$ form the left-right corner of some  box in the ladder diagram, 

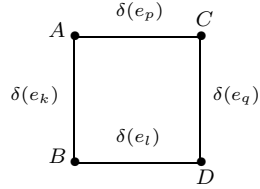
\unitlength=0.28mm

\begin{picture}(10,120)(-100,0)
\tiny

\qbezier[4000](20,30)(20,60)(20,90)
\qbezier[4000](20,30)(50,30)(80,30)
\qbezier[4000](80,30)(80,60)(80,90)
\qbezier[4000](20,90)(50,90)(80,90)

\put(18,28){\scriptsize $\bullet$}
\put(18,88){\scriptsize $\bullet$}
\put(78,28){\scriptsize $\bullet$}
\put(78,88){\scriptsize $\bullet$}

\put(-10,60){$\delta(e_k)$}
\put(40,100){$\delta(e_p)$}
\put(40,40){$\delta(e_l)$}
\put(86,60){$\delta(e_q)$}

\put(8,90){\scriptsize $A$}
\put(8,30){\scriptsize $B$}
\put(78,20){\scriptsize $D$}
\put(78,95){\scriptsize $C$}
\end{picture}\\
then we have $\delta(e_k)+\delta(e_l)=\delta(e_p)+\delta(e_q)$. Thus $H$ is zero on $\delta(e_s)$ for $s=k,l,p,q$ and it is not a supporting hyperplane. This leads to a contradiction, so $\delta(e_k)$ and $\delta(e_l)$ are contained in a smooth cone in $\Sigma'$. There are no integral points other than $\delta(e_k)$ and $\delta(e_l)$ on $\Theta'$. Hence we finish the proof of the claim.

\subsubsection{The moduli of CY hypersurfaces in $X$}
Recall that $X=SL_4/P$, where $P$ is a parabolic subgroup, From Bott's theorem \cite{B1957}, $H^q(X,T_X)=0$ for $q\geq 1$. For a smooth hypersurface $i:X_f\hookrightarrow X$, we have the normal bundle exact sequence
\[
0\to T_{X_f}\to i^*T_X \to \sN_{X_f/X}\to 0.
\]
Since $X_f$ is an anti-canonical hypersurface and thus $\sI_{X_f}\cong \sO(K_X)$, by Serre duality and Hodge theory, $H^2(X, \sI_{X_f}\otimes T_X)\cong H^2(X,\Omega_X)^*=0$. By the exact sequence
\[
0\to \sI_{X_f}\to \sO_X \to i_*\sO_{X_f}\to 0
\] 
and $i_*i^*T_X\cong T_X\otimes i_*\sO_{X_f}$, we get that $H^1(X_f,i^*T_X)=H^1(X,i_*i^*T_X)=0$. Hence the connecting homomorphism
\[
H^0(X_f,\sN_{X_f/X})\to H^1(X_f,T_{X_f})
\]
is surjective. 

\bibliographystyle{plain}

\end{document}